\documentclass[a4paper]{amsart}

\usepackage{amssymb,amscd}
\usepackage{bbm}

\usepackage{pstricks,epsfig}
\usepackage[totalwidth=17cm,totalheight=24cm]{geometry}
\usepackage{graphicx}
\usepackage{psfrag}

\usepackage[all]{xy}

\newtheorem{cor}{Corollary}[section]
\newtheorem{theorem}[cor]{Theorem}
\newtheorem{prop}[cor]{Proposition}
\newtheorem{lemma}[cor]{Lemma}

\theoremstyle{definition}
\newtheorem{defi}[cor]{Definition}
\theoremstyle{remark}
\newtheorem{remark}[cor]{Remark}

\newtheorem{conj}[cor]{Conjecture}
\newtheorem*{notation}{Notation}

\newcommand{\cC}{{\mathcal C}}

\newcommand{\cF}{{\mathcal F}}

\newcommand{\cT}{{\mathcal T}}
\newcommand{\cQ}{{\mathcal Q}}
\newcommand{\cW}{{\mathcal W}}
\newcommand{\cY}{{\mathcal Y}}
\newcommand{\cDY}{{\mathcal{DY}}}
\newcommand{\cCP}{{\mathcal C\mathcal P}}
\newcommand{\cML}{{\mathcal M\mathcal L}}
\newcommand{\cFML}{{\mathcal F\mathcal M\mathcal L}}

\newcommand{\C}{{\mathbb C}}
\newcommand{\HH}{{\mathbb H}}
\newcommand{\PP}{\mathbb{P}}
\renewcommand{\Pr}{\mathbb{P}}
\newcommand{\N}{{\mathbb N}}
\newcommand{\R}{{\mathbb R}}

\newcommand{\hd}{\dot{h}}

\newcommand{\gb}{\overline{g}}

\newcommand{\Hyp}{\mathbb{H}}
\newcommand{\AdS}{\mathbb{A}\mathbbm{d}\mathbb{S}}
\newcommand{\dS}{\mathbbm{d}\mathbb{S}}

\newcommand{\SL}{\mathrm{SL}}
\newcommand{\PSL}{\mathrm{PSL}}
\newcommand{\dual}{\star}
\newcommand{\ph}{\varphi}

\newcommand{\Dom}{\mathrm{Dom}}

\newcommand{\Diffeo}{\mbox{Diffeo}}

\newcommand{\St}{\tilde S}
\newcommand{\tr}{\mbox{\rm tr}}

\newcommand{\grad}{\operatorname{grad}}%{\mbox{grad}}
\newcommand{\isom}{\mathrm{Isom}}

\newcommand{\II}{I\hspace{-0.1cm}I}
\newcommand{\III}{I\hspace{-0.1cm}I\hspace{-0.1cm}I}

\newcommand{\rar}{\rightarrow}
\newcommand{\lra}{\longrightarrow}

\newcommand{\da}{da}

\def\l{\lambda}

\def\Hess{\mathrm{Hess}}
\def\Div{\mathrm{div}}
\def\Teich{\mathcal{T}}

\def\wti#1{\widetilde{#1}}
\def\what#1{\widehat{#1}}
\def\ol#1{\overline{#1}}

\def\Earth{\mathcal E}
\def\En{\mathbbm{1}}
\def\Ene{E}
\def\Land{\mathcal{L}}
\def\hend{M}
\def\Ends{\mathfrak{E}}
\def\des{M^\dual}
\def\mgh{N}
\def\Sy{\mathcal{S}}
\def\Sch{Sch} % {{\text{\boldmath$S$}}}
\def\qd{\varphi}

\def\wolf{\mathcal{W}}
\def\nor{n}
\def\th{\theta}
\def\skew{{r}}
\def\cG{\mathcal{G}}
\def\d{\delta}
\def\Id{id} % identity map id:S-->S
\def\SW{SW} % Sampson-Wolf's map
\def\Extr{\mathrm{Ext}}
\def\const{R} % constant >1
\def\Dom{C} % domain of dependence
\def\tDom{\tilde{C}} % univ.cover of domain of dependence
\def\cDC{\mathcal{DC}}



\begin{document}

\title{A cyclic extension of the earthquake flow II}
\author{Francesco Bonsante}
\address{Universit\`a degli Studi di Pavia\\
Via Ferrata, 1\\
27100 Pavia, Italy}
\email{francesco.bonsante@unipv.it}
\thanks{F.B. is partially supported by the A.N.R. through project Geodycos.}
\author{Gabriele Mondello}
\address{Universit\`a di Roma ``La Sapienza'' - Dipartimento di Matematica
``Guido Castelnuovo'' \\
piazzale Aldo Moro 5 \\
00185 Roma, Italy}
\email{mondello@mat.uniroma1.it}
\author{Jean-Marc Schlenker}
\thanks{J.-M. S. was partially supported by the A.N.R. through projects
ETTT, ANR-09-BLAN-0116-01, and GeomEinstein, ANR-09-BLAN-0116-01.}
\address{Institut de Math\'ematiques de Toulouse, UMR CNRS 5219 \\
Universit\'e Toulouse III \\
31062 Toulouse cedex 9, France}
\email{schlenker@math.univ-toulouse.fr}

\date{August 2012 (v1)}

\begin{abstract}
The landslide flow, introduced in [5], is a smoother analog of the earthquake
flow on Teichm\"uller space which shares some of its key properties. We show here
that further properties of earthquakes apply to landslides. The landslide flow is
the Hamiltonian flow of a convex function. The smooth grafting map $sgr$ taking values in
Teichm\"uller space, which is to landslides as grafting is to earthquakes, 
is proper and surjective with respect to either of its variables. 
The smooth grafting map $SGr$ taking values in the space of complex projective structures
is symplectic (up to a multiplicative constant). 
The composition of two landslides has a fixed point on Teichm\"uller
space. As a consequence we obtain new results on constant Gauss curvature surfaces
in 3-dimensional hyperbolic or AdS manifolds. We also show that the landslide flow has
a satisfactory extension to the boundary of Teichm\"uller space.
\end{abstract}

\maketitle

\tableofcontents

\section{Introduction and results}

In this paper we consider a closed surface $S$ of genus at least $2$. We 
denote by $\cT$ the Teichm\"uller space of $S$, considered either as the space
of hyperbolic structures on $S$ (considered up to isotopy) or as the space of
conformal structures on $S$ (also up to isotopy). We denote by $\cML$ the
space of measured laminations on $S$. 

\subsection{Earthquakes and landslides}

Let $\gamma$ be a simple closed curve on $S$, with a weight $w>0$, and let $h$ be a
hyperbolic metric on $S$. The image of $h$ by the (left) earthquake along the 
weighted curve $w\gamma$ is obtained by realizing $\gamma$ as a closed geodesic in
$(S,h)$, cutting $S$ open along this geodesic, rotating the right-hand side
by a length $w$ in the positive direction, and gluing back. This defines 
a map $\Earth(\bullet,w\gamma):\cT\rightarrow \cT$. Thurston \cite{thurston-earthquakes} proved that this definition
extends from weighted curves to measured laminations, so that we obtain a map:
$$ \Earth:\cT\times\cML\rightarrow \cT~. $$

This earthquake map has a number of remarkable properties, of which we can
single out, at this stage, the following.
\begin{enumerate}
\item For fixed $\lambda\in \cML$, it defines a flow on $\cT$: for all $t_1,t_2\in\R$, 
$\Earth(h,(t_1+t_2)\lambda)=\Earth(\Earth(h,t_1\lambda),t_2\lambda)$.
\item Thurston's Earthquake Theorem (see \cite{kerckhoff,mess}): for any
$h,h'\in \cT$, there exists a unique $\lambda\in \cML$ such that $\Earth(h,\lambda)=h'$.
\item McMullen's complex earthquakes \cite{mcmullen:complex}: for fixed $\lambda\in \cML$
and $h\in \cT$, the map $t\mapsto \Earth(h,-t\lambda)$ extends to a holomorphic map from 
the upper half-plane to $\cT$. 
\item $\Earth(h,(t+is)\lambda)=gr(\bullet,s\lambda)\circ \Earth(\bullet,-t\lambda)$, where $gr(\bullet,s\lambda):\cT\rightarrow \cT$ is 
the grafting map.
\item The grafting map $gr:\cT\times\cML\rightarrow \cT$ can be written as the composition $gr=\Pi\circ
Gr$, where $Gr:\cT\times\cML \rightarrow \cCP$ is also called the grafting map but with 
values in the space $\cCP$ of complex projective structures on $S$, and $\Pi:\cCP\rightarrow
\cT$ is the forgetful map sending a complex projective structure to the underlying complex
structure.
\item Thurston proved that the map $Gr:\cT\times\cML\rightarrow \cCP$ is a homeomorphism 
(see \cite{kamishima-tan} for a proof).
\end{enumerate}

In \cite{cyclic} we introduced the notion of {\it landslides}, which can be considered
as smooth version of earthquakes. The landslide map $\Land:S^1\times\cT\times \cT\rightarrow 
\cT\times \cT$, can be defined in different ways, see below. In \cite{cyclic} we showed
that properties (1)-(6) above extend from earthquakes to landslides, with the grafting
maps $gr$ and $Gr$ replaced by the corresponding smooth grafting maps $sgr'$ and $SGr'$. 

Here we further consider the properties of the earthquake and grafting maps.
\begin{enumerate}
\setcounter{enumi}{6}
\item For a fixed measured lamination $\lambda$, the earthquake flow $(t,h)\mapsto \Earth(h,t\lambda)$
is the Hamiltonian flow of the length function of $\frac{1}{2}\lambda$, considered as a function on $\cT$, with respect to the Weil-Petersson symplectic structure.
\item The length of a measured lamination is a convex function on $\cT$ with respect to the Weil-Petersson
metric.
\item Given two measured laminations $\lambda,\mu\in \cML$ which fill $S$, the composition $\Earth(\bullet,\lambda)\circ
\Earth(\bullet,\mu):\cT\rightarrow \cT$ has a fixed point (conjectured to be unique).
\item For fixed $\lambda$, the map $gr(\bullet,\lambda):\cT\rightarrow \cT$ is a homeomorphism (see  \cite{scannell-wolf}), while,
for $h\in \cT$ fixed, the map $gr(h,\bullet):\cML\rightarrow\cT$ is a homeomorphism (see \cite{dumas-wolf}).
\item The cotangent space $T^*\cT$ can be identified with the product $\cT\times \cML$ through the map
$d\ell:\cT\times \cML \to T^*\cT$ which sends $(h,\lambda)$ to the differential at $h$ of $d\ell_\lambda$ --- in
particular this map is one-to-one.
\item The grafting map $Gr:\cT\times \cML\rightarrow \cCP$ can be composed with 
$(d\ell)^{-1}:T^*\cT\to \cT\times \cML$ to obtain a map from $T^*\cT$ to $\cCP$. This map is actually $C^1$
(although $\cML$ does not have a natural $C^1$-structure) and symplectic, when one considers on $\cCP$ the real 
symplectic structure equal to the real part of the Goldman symplectic structure on the space of
representations of $\pi_1(S)$ to $\PSL(2, \C)$.
\end{enumerate}
We will prove that those properties extend to landslides, except for point (10) for which we only prove here
that the corresponding maps in the landslide setting are onto. (We also believe that those maps are
one-to-one, but could not prove it.)

We will 
%%also 
see that points (9) and (10) can be translated in terms of 3-dimensional 
hyperbolic or anti-de Sitter geometry. 

%%We will also show that 
In addition we will show (see Section \ref{ssc:extension} for a more precise statement) that
\begin{enumerate}
\setcounter{enumi}{12}
\item the landslide map has a satisfactory extension to the space $\cFML$ of filling pairs of measured laminations
on $S$, considered as a boundary of $\cT\times \cT$, and this extension is Hamiltonian for the symplectic structure
equal to the sum of the Thurston symplectic forms on the two factors.
\end{enumerate}

\subsection{The landslide flow is Hamiltonian}

We will first define a function $F$ on $\cT\times \cT$ that plays for landslides the role that the length
of a measured lamination plays for earthquakes. Recall that given two hyperbolic metrics $h$ and $h^\dual$ on
$S$, there is a unique minimal Lagrangian map $m$ isotopic to the identity from $(S,h)$ to $(S, h^\dual)$
(see \cite{L5,schoen:role}). This map can be characterized by the existence of a bundle morphism $b:TS\rightarrow TS$
%% I've added the references -- I hope they're the correct ones.
which has determinant $1$, is self-adjoint for $h$ and satisfies the Codazzi equation $d^\nabla b=0$, 
and such that $m^*h^\dual=h(b\bullet,b\bullet)$. We call $b$ the {\it Labourie operator} of the pair $(h,h^\dual)$
and $c$ the {\it center} of $(h,h^\dual)$, namely the conformal structure (up to isotopy) underlying the metric $h+m^* h^\dual$.

\begin{defi}
Let $F:\cT\times \cT\rightarrow \R$ the function defined as 
$$ F(h,h^{\dual}) = \int_S \tr(b) \da_h~, $$
where $b$ is the Labourie operator of the pair $(h,h^{\dual})$
and $\da_h$ is the area element associated to $h$.
\end{defi}

Note that $F(h,h^{\dual})=F(h^{\dual},h)$: if $b$ is the Labourie operator of the pair $(h,h^\dual)$, then 
the Labourie operator of the pair $(h^\dual,h)$ is $b^\dual=b^{-1}$, and $\tr(b^\dual)=\tr(b)$ since $b$ has
determinant $1$.

\begin{prop}\label{prop:energy}
Let $c$ be the center of $(h,h^\dual)$.
The functions $\Ene(\bullet,h),F(\bullet,h):\cT\rar\R$ are proper
and real-analytic, where $\Ene(c,h)$ is the energy of the unique
harmonic map from $(S,c)$ to $(S,h)$.
Moreover, $2F(h,h^{\dual})=\Ene(c,h)=\Ene(c,h^\dual)$.
\end{prop}

\begin{theorem} \label{tm:hamiltonian}
The landslide flow on $\cT\times \cT$ is the Hamiltonian flow
associated to $\frac{1}{4}F$
for the symplectic form $\omega_{WP,1}+\omega_{WP,2}$.
\end{theorem}

As a consequence we see that the landslide flow is certainly {\it not}
the same as the Hamiltonian flow of the length of the Liouville
cycle.%, introduced earlier by one of us \cite{}.

\subsection{Convexity of the Hamiltonian}

The following result is an extension to landslides of the convexity of
the length function of measured laminations. 

\begin{theorem} \label{tm:cvx}
Let  $h\in \cT$ be fixed. The function $F(h,\bullet):\cT\rightarrow \R$ is
strictly convex for the Weil-Petersson metric on $\cT$. More precisely,
at each point $h^{\dual}$, $\Hess(F(h,\bullet))\geq 2g_{WP}$, with equality
exactly when $h=h^{\dual}$.  
\end{theorem}

%%One possible application: a new proof of the Nielsen realization problem. 
%%(Presumably this is not very exciting...)

\subsection{Landslide symmetries}

There is a simple notion of ``symmetry'' associated to the notion of landslides --- it actually
also makes some sense for earthquakes, see below.

\begin{defi}
Let $\theta\in (0,\pi)$, and let $h_0\in \cT$. For all $h\in \cT$, there is a unique
$h_0^\dual\in \cT$ such that $\Land^1_{e^{i\theta}}(h_0,h_0^\dual)=h$ (see \cite[Theorem 1.14]{cyclic}). 
We set $\Sy_{e^{i\theta},h_0}(h)=\Land^1_{e^{-i\theta}}(h_0,h_0^\dual)$. We call $\Sy$ the symmetry of center 
$h_0$ and angle $\theta$. 
\end{defi}

Note that $\Sy_{e^{i\theta},h_0}$ is not an involution, however, by definition, 
$\Sy_{e^{i\theta},h_0}\circ \Sy_{e^{-i\theta},h_0}=Id$.

The following statement is an analog for landslides of the main statement in 
\cite{earthquakes}.

\begin{theorem} \label{tm:fixed}
Let $\theta_+,\theta_-\in (0,\pi)$ and $h_+,h_-\in \cT$ be fixed. 
The map $\Sy_{e^{i\theta_+},h_+}\circ \Sy_{e^{i\theta_-},h_-}:\cT\rightarrow \cT$
has a fixed point. If $\theta_++\theta_-=\pi$ then this fixed point is unique.
\end{theorem}

It would be quite satisfactory to know whether uniqueness holds for other
values of $\theta_++\theta_-$. 

This statement can also be translated in terms of 3-dimensional AdS geometry, see below. 
The uniqueness question which remains open can then be translated as a natural
statement on the uniqueness of $\AdS^3$ manifolds with smooth, space-like boundary
having a given pair of constant curvature metrics as the induced metric on
the boundary, and the analogy with the corresponding hyperbolic situation 
suggests that it might be true.

In the limit case of earthquakes, one can define a similar notion of symmetry. 
Given a fixed $h_0\in \cT$, we can define the (left) earthquake symmetry $\Sy_{h_0}$
as follows. For any $h\in \cT$, there is by Thurston's Earthquake Theorem (see \cite{kerckhoff,mess})
a unique $\lambda\in \cML$ such that $\Earth(h_0,\lambda)=h$, and we define $\Sy_{h_0}(h)=\Earth(h_0,-\lambda)$. 
One can then ask whether, for $h_+, h_-\in \cT$, the composition $\Sy_{h_+}\circ \Sy_{h_-}$ has a unique
fixed point. A positive answer would be equivalent to a proof of a conjecture of Mess \cite{mess} on 
the existence and uniqueness of a MGH AdS manifold for which the induced metric on the boundary
of the convex core is a prescribed pair of hyperbolic metrics. We leave details on this to the reader.

\subsection{Smooth grafting}

In Section \ref{sc:sgr} we turn to the smooth grafting map, defined in \cite{cyclic}
and recalled in Section \ref{ssc:smooth-grafting}, %% [I've added this]
which is a smoother analog of the grafting map.
We have the following partial extension/analog of a result of Scannell and Wolf 
\cite{scannell-wolf} and of a result of Dumas and Wolf \cite{dumas-wolf}.

\begin{theorem} \label{tm:sgr}
Let $s>0$ and $h,h^\dual\in \cT$. The maps $sgr'_s(h,\bullet):\cT\rightarrow \cT$
and $sgr'_s(\bullet,h^\dual):\cT\rightarrow \cT$ are proper surjective maps.
\end{theorem}

This result can be stated in terms of 3d hyperbolic or de Sitter geometry.
Recall that any hyperbolic end has a unique foliation by constant curvature surfaces
\cite{L5}. The curvature of those surfaces varies between $-1$ and $0$.

\begin{theorem} \label{tm:hyperbolic}
Let $h,h'\in \cT$ and let $K\in (-1,0)$. 
There is a hyperbolic end $\hend$ with conformal metric at infinity $h'$ and such that
the surface $S^\dual$ in $\hend$ with constant curvature $K$ has an induced metric proportional to $h$. 
\end{theorem}

It would be satisfactory to know whether $\hend$ is unique. 

Similarly, any 3-dimensional de Sitter domain of dependence (as defined in \cite{mess}) has
a unique foliation by constant curvature surfaces, which are actually dual to the constant
curvature surfaces in the foliation of the dual hyperbolic end (see \cite{BBZ2}). 

\begin{theorem} \label{tm:deSitter}
Let $h^\dual,h'\in \cT$ and let $K^\dual\in (-\infty,0)$. 
There is a de Sitter domain of dependence $\des$ with conformal metric at infinity $h'$ and such that
the surface $S$ in $\des$ with constant curvature $K^\dual$ has an induced metric homothetic to $h^\dual$. 
\end{theorem}

The proof of those 3-dimensional theorems, from Theorem \ref{tm:sgr}, are in Section \ref{ssc:ends}

\subsection{The smooth grafting map is symplectic}

Consider a measured lamination $\lambda\in \cML$. Its length $\ell_\lambda$ is a smooth function
on $\cT$ and, for each $h\in \cT$, we can consider its differential $d_h\ell_\lambda\in T^*_h\cT$. 
It is well-known that this defines a one-to-one map between $\cT\times \cML$ and $T^*\cT$. 
This has the following counterpart in the context considered here, with $\cML$ replaced by another
copy of $\cT$ and the length function replaced by the function $F$ defined above. 

\begin{prop} \label{pr:parameterization}
%% Let $h\in \cT$. 
For all $h^\dual\in \cT$, let  
$$ \begin{array}{llcl}
  F_{h^\dual}: & \cT & \to & \R \\
 & h & \mapsto & F(h ,h^\dual)~. %%[I've replaced h' by h here] 
\end{array} $$
The map $d_1F:\cT\times \cT\to T^*\cT$ sending $(h,h^\dual)$ to the differential at $h$ of 
the function $F_{h^\dual}$ is a global diffeomorphism between $\cT\times \cT$ and $T^*\cT$.
\end{prop}

Coming back to the familiar setting of measured lamination, we can consider the 
grafting map $Gr:\cT\times \cML\rightarrow \cCP$, which is known from Thurston's work
to be a homeomorphism (see \cite{kamishima-tan}). 
Composing with the inverse of the map $d\ell:\cT\times \cML\rightarrow
T^*\cT$, we obtain a map $Gr\circ (d\ell)^{-1}:T^*\cT\rightarrow \cCP$. 

Given a complex projective structure $\Xi$ on a surface, one can consider the underlying 
complex structure, say $c$. Riemann uniformization produces another $\C \PP^1$-structure $\Xi_F$
on $S$ which is Fuchsian with underlying complex structure $c$. We can then consider the
Schwarzian derivative of the identity map from $(S,\Xi_F)$ to $(S,\Xi)$, it is a holomorphic
quadratic differential $\qd$ for $c$. Its real part can therefore be considered as a vector in 
$T^*_c\cT$. This classical construction defines a map $\Sch:\cCP\to T^*\cT$. It was proved by Kawai
\cite{kawai} that this map is symplectic (up to a factor), that is, the pull-back by $\Sch$ of the
(real) cotangent symplectic structure on $T^*\cT$ is a multiple of the real part of the Goldman
symplectic form $\omega_G$ on $\cCP$.

We can now consider the map $\Sch\circ Gr\circ (d\ell)^{-1}:T^*\cT\rightarrow T^*\cT$.
It is proved in \cite{cp}
that this map is $C^1$, and that it is symplectic (up to a fixed factor) with respect to the cotangent symplectic
form on $T^*\cT$. Here we denote by $\omega_{can}$ the (real) cotangent symplectic form on $T^*\cT$.

Now recall from \cite{cyclic} the definition of the smooth grafting map 
$SGr':\R\times \cT\times \cT\to \cCP$. For $s\geq 0$ and $h,h^\dual\in \cT$, there is 
a unique equivariant convex immersion of the universal cover $\St$ of $S$ into $\Hyp^3$ with induced metric $\cosh^2(s/2)h$
and third fundamental form $\sinh^2(s/2)h^\dual$. This equivariant immersion defines on $S$
a complex projective structure, obtained by pulling back on the image surface the complex
projective structure at infinity by the ``Gauss map'' sending a point $x$ of the image to the
endpoint at infinity of the geodesic ray from $x$ orthogonal to the surface.
This complex projective structure is the image $SGr'_s(h,h^\dual)$. 

\begin{theorem} \label{tm:sgr-symplectic}
For all $s\in \R$, the composition map $\Sch\circ SGr'_s\circ (dF_1)^{-1}:T^*\cT\to T^*\cT$ is symplectic up to a constant
factor depending on $s$: 
$$ (\Sch \circ SGr'_s\circ (dF_1)^{-1})^*\omega_{can} = \sinh(s) \omega_{can}~. $$
\end{theorem}

The proof uses a variant of the notion of renormalized volume, generalizing the definition introduced for a 
similar purpose in \cite{cp}.

\subsection{Extension at the boundary} \label{ssc:extension}

In order to understand how the landslide flow could be
extended to the boundary, let's recall the following result by Wolf.

\begin{prop}[\cite{wolf:teichmuller}]
Let the center $c$ stay fixed and suppose that $\th_n h_n\rightarrow \lambda$.
Then $\th_n h_n^{\dual}\rightarrow \mu$, with the property that the
unique holomorphic quadratic differential $\qd$ with horizontal and vertical laminations
$\lambda$ and $\mu$ has $c$ as underlying complex structure.
\end{prop}

If we consider the space $\cDY=\cT\times \cT\times \R_{<0}$ 
of couples of metrics $(h,h^{\dual})$ with the same negative
constant curvature $K$ (up to isotopy), then this space can be bordified as $\ol{\cDY}$ by adding
the space $\cFML\subset\cML\times \cML$ of filling couples of measured laminations correspondingly to $K=-\infty$.

\begin{prop}\label{prop:action}
Identifying $\cFML$ with the space $\cQ$
of holomorphic quadratic differentials on $S$,
the action on $\cDY$ limits on $\cFML=\cQ$ to the action
$(\theta,\qd)\mapsto e^{i\theta}\qd$.
\end{prop}

Notice that the function $F$ extends to $\cDY$ as
$F(h,h^\dual,K):=K^{-2}F(h,h^\dual)$.
As the Weil-Petersson symplectic form $\omega_{WP}$
limits to Thurston's symplectic form $\omega_{Th}$ on $\cML$,
Theorem~\ref{tm:hamiltonian} admits the following extension.

\begin{prop}\label{prop:ham-ext}
The function $F:\cDY\rar\R_+$ extends to $\partial{\cDY}=\cFML$ as
$F(\lambda,\mu)=i(\lambda,\mu)$.
% or is it $(1/2)i(\lambda,mu)$ ?
Moreover, the extension of the landslide flow on $\partial\cDY=\cFML$ is
Hamiltonian for the symplectic form $\omega_{Th,1}+\omega_{Th,2}$ with respect to $\frac{1}{4}F$.
\end{prop}

\subsection{Constant curvature surfaces in globally hyperbolic $\AdS^3$ manifolds}

Recall that given a globally hyperbolic $\AdS^3$ manifold $\mgh$, the complement of its
convex core has a unique foliation by constant curvature surfaces \cite{BBZ2}. The 
curvature of those surfaces varies between $-\infty$ and $-1$.

\begin{theorem} \label{tm:adsI}
Let $h_+,h_-\in \cT$, and let $K_+,K_-\in (-\infty, -1)$. 
There exists a globally hyperbolic $\AdS^3$ manifold $\mgh$ such that the constant curvature
$K_-$ surface in the past of the convex core has induced metric homothetic to $h_-$,
while the constant curvature $K_+$ metric in the future of the convex core of constant
curvature $K_+$ has induced metric homothetic to $h_+$. If $K_+=-K_-/(K_-+1)$, then $\mgh$
is unique.
\end{theorem}

It is tempting to conjecture that $\mgh$ is unique for any value of $K_+$ and $K_-$. 
Actually the analogy with quasifuchsian hyperbolic 3-manifolds indicates that 
Theorem \ref{tm:adsI} could well extend to metrics of non-constant curvature, as
follows (see \cite[\S 3.4]{adsquestions} for more on this). %% [added this parenthesis]

\begin{conj} \label{cj:adsI}
Let $h_+, h_-$ be two smooth metrics on $S$ with curvature $K<-1$. There exists a
unique globally hyperbolic $\AdS^3$ manifold 
homeomorphic to $S\times [-1,1]$,
with smooth, space-like and strictly convex
boundary, such that the induced metric on 
$S\times \{ -1\}$ is $h_-$ and the induced metric on $S\times \{ 1\}$ is $h_+$.
\end{conj}

Theorem \ref{tm:adsI} shows that the existence part of this statement holds when
both $h_-$ and $h_+$ have constant curvature.
The analog 
%%statement 
of Conjecture \ref{cj:adsI} 
for quasifuchsian manifolds (and more generally convex co-compact
manifolds) holds, see \cite{hmcb}.

As the limit case of either Conjecture \ref{cj:adsI} or Theorem \ref{tm:adsI} when
the curvature of the metrics $h_-$ and $h_+$ goes to $-1$, we obtain the following conjecture
of Mess.

\begin{conj} \label{cj:mess}
Let $h_-, h_+$ be two hyperbolic metrics on $S$. There is a unique maximal globally hyperbolic
$\AdS^3$ manifold $\mgh$ such that induced metric on the boundary of the convex core of $\mgh$ is
given by $h_-$ and $h_+$.
\end{conj}

There is a useful notion of duality in $\AdS^3$, recalled in Section 2. It can be used to translate
Theorem \ref{tm:adsI} in terms of the third fundamental form, rather than the induced metric,
on surfaces in $\AdS^3$ manifolds. 

\begin{theorem} \label{tm:adsIII}
Let $h_+,h_-\in \cT$, and let $K_+,K_-\in (-\infty, -1)$. 
There exists a maximal globally hyperbolic $\AdS^3$ manifold $\mgh$ containing a surface $S_-$ with
third fundamental form of constant curvature $K_-$ homothetic to $h_-$ in the past
of the convex core, and a surface $S_+$ with constant curvature $K_+$ and third 
fundamental form homothetic to $h_+$ in the future of the convex core. 
If $K_+=-K_-/(K_-+1)$, then $\mgh$ is unique.
\end{theorem}

As for Theorem \ref{tm:adsI}, the analogy with quasifuchsian manifolds indicates 
that the statement might hold also for metrics of variable curvature --- this would
actually follow from Conjecture \ref{cj:adsI} using the same notion of duality.

\subsection{Content of the paper}

Section 2 contains notations and background material, including previous definitions
and results on the landslide flow considered here, on smooth grafting, and on their
relationships with 3d hyperbolic and $\AdS$ geometry. 

In Section 3 and 4 we will describe the Hamiltonian interpretation of the landslide flow,
Theorem \ref{tm:hamiltonian},
and relate its Hamiltonian to the energy of underlying harmonic maps. In Section 5
--- probably the most technically involved part of the paper --- we show that this
Hamiltonian function is convex for the Weil-Petersson metric on $\cT$, Theorem
\ref{tm:cvx}. 

In Section 6 we turn to the smooth grafting, and prove Theorem \ref{tm:sgr}
and then, as an application, Theorem \ref{tm:hyperbolic}. The symplectic properties
of the smooth grafting map are investigated in Section 7. In 
Section 8 we consider the extension of the landslide flow to the boundary. 
Finally in Section 9 we prove Theorem \ref{tm:adsI} on 3-dimensional 
$\AdS$ geometry and Theorem \ref{tm:fixed} on fixed points of compositions of landslides.

\section{Notations and background material}

%% [I've added the following introductory paragraph]
This section collects a number of definitions and results which are used below. It is included
here so as to make the paper as self-contained as possible.

\subsection{The space of complex structures and Weil-Petersson product}

We fix an oriented closed surface $S$ of genus $g(S)\geq 2$.
The Teichm\"uller space of $S$ is the quotient of the space of complex structures on $S$ 
by the action of $\Diffeo_0(S)$.%% [Here and elsewhere I've used this macro instead of just writing Diffeo]
By the uniformization theorem, it can be also regarded as the space of hyperbolic
metrics on $S$ up to isotopy.

We will denote by $\mathcal A$ the space of almost-complex structures on $S$: in other words,
an element of $\mathcal A$ is an operator $J$ on $TS$ such that
\begin{itemize}
\item  $J^2=-\En$. 
\item For every $0\neq X\in T_pS$, the basis $(X, JX)$ is %% positive.
positively oriented.
\end{itemize}
Since in dimension $2$ every
almost-complex structure is integrable, we have a natural map
\[
    \mathcal A\rightarrow \mathcal T,
\]
which is a $\Diffeo_0(S)$-principal bundle, where $\Diffeo_0(S)$ acts on $\mathcal A$ 
by pull-back (\cite{fischer-tromba:cell}).

Let $\mathcal M_{-1}$ the space of hyperbolic metrics on $S$.
By the %% uniformization theorem
Uniformization Theorem, 
the map $\mathcal M_{-1}\rightarrow\mathcal A$ sending
$h$ to the complex structure compatible with $h$ is a $\Diffeo(S)$-equivariant identification.
In particular the %% elemnts
elements 
of $\mathcal T$ can be also regarded as hyperbolic metrics
up to isotopies.

Let $h_0$ be a hyperbolic metric on $S$ and denote by $J_0$ the corresponding
almost-complex structure. %compatible with the orientation of $S$.

The tangent space $T_{J_0}\mathcal A$ is the set of operators $\dot J$ on $TS$ such that
\[
    \dot JJ_0+J_0\dot J=0~.
\]
Notice that there are two simple characterizations of $\dot J$:
\begin{itemize}
\item it is a $\mathbb C$-anti-linear operator,
\item it is traceless and $h_0$-self-adjoint.
\end{itemize}

The tangent space of $\mathcal T$ at $[J_0]$  
turns out to be identified to the quotient of $T_{J_0}\mathcal A$ by the vertical subspace.
We want to relate this description of $T_{[J_0]}\mathcal T$ with the classical
description in terms of Beltrami differentials.
We will give a description of quadratic differentials and Beltrami differentials
as tensors on the surface.

Let us fix a complex atlas $\{(U_j, z_j)\}$ on $S$ compatible with $J_0$. Namely,
 putting $z_j=x_j+iy_j$ it results that %% [added "that"]
\[
   J_0\frac{\partial\,}{\partial x_j}=\frac{\partial\,}{\partial y_j}\,,\qquad
   J_0\frac{\partial\,}{\partial y_j}=-\frac{\partial\,}{\partial x_j}~.
\]

\subsubsection{Holomorphic quadratic differentials}
A holomorphic quadratic differential $\qd$ on $S$ is a holomorphic
section of the square of the canonical bundle $(\Omega^{1,0}S)^{\otimes 2}$: in local coordinates
$\qd|_{U_k}=\qd_j(z_j)dz_j^2$, where on the intersection
$U_j\cap U_k$ we have
\[
    \qd_j(z_j)=\qd_k(z_k)\left(\frac{dz_k}{dz_j}\right)^2(z_k)~.
\]
Quadratic differentials can be regarded as complex bilinear
tensors: indeed given $p\in U_j$ and $X,X'\in T_pS$ we can decompose
\[
   X=u\frac{\partial\,}{\partial x_j}+v\frac{\partial\,}{\partial y_j}\,,\qquad
   X'=u'\frac{\partial\,}{\partial x_j}+v'\frac{\partial\,}{\partial y_j}\,.
\]
Then, setting
\[
  \qd(X,X')=\qd_j(z_j(p))(u+iv)(u'+iv')
\]
one can directly check that this definition does not depend on the
complex chart $z_j$ and %% define 
defines
a complex bilinear form at $T_pS$.

Holomorphic %%harmonic 
quadratic
differentials form a complex vector 
space --- denoted here by $\mathcal Q(J_0)$  --- of complex dimension
$3g(S)-3$.

There is a holomorphic vector bundle $\pi:\cQ\rightarrow\cT$
whose fiber on the point $[J]$ is canonically identified with $\cQ(J)$.
This is called the fiber bundle of holomorphic quadratic differentials.

\subsubsection{Beltrami differentials}
A Beltrami differential $\nu$ is a section of the differentiable linear
bundle $\Omega^{-1,1}S$: 
locally it can be written  as %% [added "as"]
\[
   \nu|_{U_j}=\nu_j(z_j)\frac{d\bar z_j}{d z_j}~,
\] 
and on $U_j\cap U_k$ we have 
 \[
     \nu_j(z_j)=\nu_k(z_k)\frac{\left(\overline{\frac{dz_k}{dz_j}}\right)}{\left(\frac{dz_k}{dz_j}\right)}(z_j)~.
 \]
 
Beltrami differentials can be regarded as 
anti-linear operators of $TS$.
Indeed, given $p\in U_j$ and $X\in T_pS$ with
$X=u\frac{\partial\,}{\partial x_j}+v\frac{\partial\,}{\partial y_j}$
we can put $\nu(X)= t\frac{\partial\,}{\partial
  x_j}+s\frac{\partial\,}{\partial y_j}$ where
$t+is=\nu_j(z_j(p))\overline{(u+iv)}$. 

 The matrix representative of %% [added "of"]
$\nu$ with respect to the real basis
 $\{\frac{\partial\,}{\partial x_j},\frac{\partial\,}{\partial y_j}\}$ is 
 \[
    [\nu]_j=\begin{pmatrix} \Re \nu_j & \Im \nu_j\\ \Im \nu_j & -\Re\nu_j\end{pmatrix}~,
 \]
whereas $|\nu|^2=\frac{1}{2}\tr(\nu^2)$.
Finally notice that the multiplication by $i$ of Beltrami differentials
corresponds to the composition with the complex structure on $TS$:
\[
   (i\nu)(X)=J_0\nu(X)~.
\]

There is a classical open embedding of $\mathcal A$ into
$Belt(J_0)$, that can be described in the following way.
Given a complex structure $J$, the identity map
\[
  \En: (T_pS, J_0)\rightarrow (T_pS, J)
\]
decomposes into a $\mathbb C$-linear part $\partial \En$
and an anti-linear part $\bar\partial \En$: namely
\[
   \partial \En=\frac{\En-JJ_0}{2}\qquad \bar\partial \En=\frac{\En+JJ_0}{2}~.
\]
In particular, the operator
\[
  \nu_J=(\partial \En)^{-1}\circ(\bar \partial \En)
\]
is an anti-linear operator of $(S, J_0)$ and is thus a Beltrami differential.
Locally around $p$, if $z$ is a local complex coordinate for $(S, J_0)$ and
$w$ is a local complex coordinate for $(S, J)$, we have
\[
   \nu_J=\frac{\frac{\partial w}{\partial \bar z}}{\frac{\partial w}{\partial z}}\frac{d\bar z}{dz}~.
\]

It is a classical fact that the map
\[
 \mathcal A\ni J\mapsto \nu_J\in Belt(J_0)
\]
is an open embedding whose image is the space of Beltrami differential with
$L^\infty$-norm less than $1$ (\cite{}). 
%% [should we give a reference for this??]

In this way any Beltrami differential corresponds to an infinitesimal deformation
of the complex structure of $J_0$. We say that a Beltrami differential is trivial
if it corresponds to a trivial deformation of the %%compelx 
complex structure. 
Trivial Beltrami differentials form a vector space which we denote by
$Belt_{tr}(J_0)$.

Classically the tangent space of Teichm\"uller space is identified to the quotient
\[
    T_{[J_0]}\mathcal T= Belt(J_0)/Belt_{tr}(J_0)~.
\]
Notice that $[\dot\nu]$ is a tangent vector to a curve $[J_t]$ of complex structures
(starting from $J_0$) if $\nu_{J_t}=t\dot\nu+o(t)$.

In particular, differentiating the relation $\nu_{J_t}=(\En-J_tJ_0)^{-1}(\En+J_tJ_0)$,
we get
\begin{equation}\label{eq:belt}
  \dot\nu=\frac{1}{2}\dot JJ_0=-\frac{1}{2}J_0\dot J
\end{equation}
It turns out that the differential of the natural projection
$\pi:\mathcal A\rightarrow\mathcal T$ at $J_0$
is simply the map
\[
d\pi:T_{J_0}\mathcal A\ni \dot J\mapsto [-\frac{1}{2}J_0\dot J]\in Belt(J_0)/Belt_{tr}(J_0)~.
\]

\subsubsection{Pairing between quadratic differentials and Beltrami differentials}
Given a holomorphic quadratic differential $\qd$ and a Beltrami
differential $\nu$ we can consider the complex-bilinear form
$\qd\bullet\nu$ which is defined on $U_j$ as
\[
   \qd\bullet\nu|_{U_j}=\qd_j(z_j)\nu_j(z_j) dx_j\wedge dy_j
\]
The form $\qd\bullet\nu$ can be described explicitly has an
alternating $2$-form on $T_pS$ in the following way. Given $X,X'\in
T_pS$ we have
\[
   (\qd\bullet\nu)(X,X')=\frac{\qd(\nu(X), X')-\qd(\nu(X'), X)}{2i}
\]

It is a classical fact that a Beltrami differential $\nu$
is trivial iff 
\[
   \int_S \qd\bullet\nu=0
\]
for all holomorphic quadratic differentials $\qd$.

As a consequence, the pairing $Belt(J_0)\times\mathcal
Q(J_0)\rightarrow\mathbb C$ induces %% to 
on the quotient a non-degenerate pairing
\[
   T_{[J_0]}\mathcal T\times\mathcal Q(J_0)\rightarrow\mathbb C
\]
which allows to identify $\mathcal Q(J_0)$ with the cotangent space of
$\mathcal T$ at $[J_0]$.

\subsubsection{Harmonic Beltrami differentials and Weil-Petersson metric}
Given a holomorphic Beltrami differential $\qd$, its real part is a
symmetric $2$-form on $S$, so there exists an $h_0$-self-adjoint operator
$\nu_\qd$ such that
\[
\Re(\qd)(X,X')=h_0(\nu_\qd(X), X')~.
\]
Since $\qd(J_0X, J_0 X')=-\qd(X,X')$ we deduce that $J_0\nu_\qd
J_0=\nu_\qd$, that is $\nu_\qd$ is $\mathbb C$-anti-linear. This means that
$\nu_\qd$ is a Beltrami operator. %% and  
It is called the harmonic
Beltrami operator associated to $\qd$. 

If $\qd|_{U_j}=\qd_j(z_j) dz_j^2$ and $h|_{U_j}=h_j(z_j)|dz_j|^2$ we easily see that
\begin{equation}\label{eq:harm}
   (\nu_\qd)|_{U_j}=\frac{\overline \qd_j(z_j)}{h_j(z_j)}
\frac{d\overline z_j}{dz_j}~.
\end{equation}

Since $\Im(\qd(X,X'))=-\Re(\qd(J_0X, X'))$ we get
\[
    \qd(X,X')=h_0(\nu_\qd(X),X')+i h_0(J_0\nu_\qd(X), X')~.
\] 

In particular the map
\[
\tilde{wp}: \mathcal Q(J_0)\ni\qd\mapsto\nu_\qd\in Belt(J_0)
\]
is $\mathbb C$-anti-linear and injective. Its image can be %% characterizes 
characterized as 
the set of self-adjoint traceless operators which satisfy the %% [added "the"]
Codazzi equation
$d^\nabla\nu=0$, where $\nabla$ is the Levi-Civita connection of the hyperbolic
metric $h_0$ (\cite{minsurf}).
%begin added
Let us recall that $d^\nabla\nu$ is in general a $2$-form with values in the tangent bundle
$TS$ defined by
\[
     d^\nabla\nu(u,v)=\nabla_u(\nu v)-\nabla_v(\nu u)-\nu([u,v])~.
\]
Regarding $\nu$ as a $1$-form on $S$ with values in $TS$, $d^\nabla$ coincides with the exterior
differential with respect to $\nabla$ on the tangent bundle.
%end added

The map $\tilde{wp}$ induces a sequilinear-form
\[
  \langle \qd, \qd'\rangle_{WP}=\int_S\qd\bullet\nu_{\qd'}~.
\]

By a  local check using (\ref{eq:harm}), one sees that
$\langle \bullet, \bullet\rangle_{WP}$ is a positive hermitian form. %% and it 
It is called the Weil-Petersson metric on $\mathcal Q(J_0)$.

We will denote by $g_{WP}$ the real part of $\langle \bullet, \bullet\rangle_{WP}$
--- which is the Weil Petersson product --- whereas $\omega_{WP}$ denotes the imaginary part,
which is the Weil-Petersson symplectic form. 

We deduce:
\begin{itemize}
\item The harmonic Beltrami differential $\nu_\qd$ is trivial iff $\qd=0$.
 Indeed  $\int_S \qd\bullet\nu_\qd=\|\qd\|^2_{WP}$.
\item The induced map $wp:\mathcal Q(J_0)\rightarrow T_{[J_0]}\mathcal T$
  is an anti-linear isomorphism. This means that every element in
  $T_{[J_0]}\mathcal T$ admits a unique harmonic representative.
\item Identifying $T_{[J_0]}\mathcal T$ with $\mathcal Q(J_0)^*$, the
  map $\qd\mapsto \langle \bullet,
  \qd\rangle_{WP}$ induced by the Weil-Petersson metric coincides with $wp$.
%% [I've displaced the equation] 
\item In particular we get that the Weil Petersson metric on the tangent space
of $\mathcal T$ is simply:
\[
   \langle [\nu_\qd], [\nu_{\qd'}]\rangle_{WP}=\langle \qd, \qd'\rangle_{WP}~.
\]
\end{itemize}

A local computation shows that in general
\[
   \Re( \qd\bullet\nu)=\frac{1}{2}\int_S\tr(\nu_\qd\nu)\da_{h_0}~,
\]
where $\da_{h_0}$ is the area form of $h_0$.
In particular we deduce that
\[
\langle [\nu_\qd], [\nu_{\qd'}]\rangle_{WP}=
\langle \qd, \qd'\rangle_{WP}=
\frac{1}{2}\left(\int_S (\tr(\nu_\qd, \nu_{\qd'})+i\tr(J_0\nu_\qd\nu_{\qd'}))\da_{h_0}\right)~.
\]

\subsubsection{Fischer-Tromba product}

Given a holomorphic quadratic differential $\qd$, 
we denote by $\dot J_\qd=2J_0\nu_\qd$, the infinitesimal deformation
of $J_0$ corresponding to the Beltrami differential $\nu_\qd$.

We consider on $T_{J_0}\mathcal A$ the hermitian product
\[
   (\dot J,\dot J')_{J_0}=\int_S\tr(\dot J\dot J')\da_{h_0}
\]
where $\da_{h_0}$ is the area form of $h_0$.

\begin{lemma}\cite{fischer-tromba:cell}
The image of the map $\mathcal Q(J_0)\ni\qd\mapsto \dot J_\qd\in T_{J_0}\mathcal A$
is the orthogonal complement of the vertical space.
\end{lemma}

\begin{proof}
Notice that $\dot J$ is vertical iff $\nu_{\dot J}=-\frac{1}{2}J_0\dot J$ is trivial, this being the case exactly iff
for all $\qd\in\mathcal Q(J_0)$ we have
\[
  0= \Re\left(\int_S\qd\bullet\nu_{\dot J}\right)=\frac{1}{2}\int_S\tr(\nu_\qd\nu_{\dot J})=
  \frac{1}{8}\int_S\tr(\dot J_\qd, \dot J)\da_{h_0}.
\]
\end{proof}

In particular given two vectors $\dot J$ and $\dot J'$ in $T_{J_0}\mathcal A$
and denoting by $\dot J_H$ and by $\dot J'_H$ their projection on  the orthogonal complement
of the vertical space, %% we have that 
there exist two quadratic differentials $\qd$ and $\qd'$
such that $\dot J_H=\dot J_{\qd}$ and $\dot{J}'_H=\dot{J}_{\qd'}$.
Since %% we have 
$d\pi(\dot J)=d\pi(\dot J_H)=[\nu_{\qd}]$ and
$d\pi(\dot J')=d\pi(\dot J'_H)=[\nu_{\qd'}]$,
 we easily deduce that
\begin{equation}\label{eq:ft1}
g_{WP}( d\pi(\dot J), d\pi(\dot J'))=\frac{1}{2}\int_S\tr(\nu_{\qd}\nu_{\qd'})\da_{h_0}
=\frac{1}{8}\int_S\tr(\dot J_H\dot J'_H)\da_{h_0}=\frac{1}{8}(\dot J_H, \dot J'_H)_{J_0}~.
\end{equation}
Analogously,
\begin{equation}\label{eq:ft2}
\omega_{WP}(d\pi(\dot J), d\pi(\dot J'))=
\frac{1}{8}(J_0\dot J_H, \dot J'_H)_{J_0}~.
\end{equation}

\subsection{Harmonic maps vs Minimal Lagrangian map}

%% [added following introductory sentence]
We collect here a number of basic facts on harmonic maps and minimal Lagrangian maps between
hyperbolic surfaces, and the relation between those two notions.

\subsubsection{Harmonic maps}
Let $h_0$ and $h$ be two  metrics on $S$. For every $C^1$ map $f:(S, h_0)\rightarrow(S, h)$,
we have the following decomposition of $f^*(h)$
\[
   f^*(h)=\qd+e h_0+\bar\qd
\]
where $\qd$ is a $J_0$-complex bilinear form on $S$,
%begin added
called the Hopf differential of $f$, 
%end added
 and $e$ is a positive function on $S$
called the energy density of the map $f$. The total energy of the map $f$ is defined as
\[
  \Ene(f)=\int_S e\,\da_{h_0}~,
\]
where $\da_{h_0}$ is the area form associated with the metric $h_0$.

We say that the map $f$ is harmonic 
%begin added
if $f$ is a stationary point of the functional $\Ene$. 
%FB: I think that this is the correct definition of harmonic
If $f$ is a diffeomorphism this is equivalent to requiring that 
%end added
 $\qd$ is a \emph{holomorphic} quadratic differential on $(S,J_0)$ \cite{}.
 %This is equivalent to %% require 
%requiring %% [I think that's grammatically correct, in principle at least]
%that $f$ is a stationary point of the functional $\Ebe(\bullet)$.
Notice that $\qd$ and $\Ene(f)$  do not change by changing $h_0$ in its conformal
class (but $e$ does), so the harmonicity of the map $f$ only depends on the complex structure on the source
surface. 

Let us fix a complex structure $J_0$ (or equivalently a hyperbolic metric $h_0$).
Given a holomorphic quadratic differential $\qd$, Wolf \cite{wolf:teichmuller} %% [added ref -- hope it is correct]
proved there exists a unique
hyperbolic metric $h_\qd$ on $S$ such that the identity $\Id:(S, J_0)\rightarrow (S, h_\qd)$
is an harmonic map with Hopf differential equal to $\qd$.
In other words, there is a unique hyperbolic metric on $S$ of the form
\begin{equation}\label{eq:harmo}
   h_\qd=\qd+eh_0+\bar\qd~.
\end{equation}

This allows to construct a map
\[
\begin{array}{rccc}
  \wolf_{J_0}: & \cQ(J_0) & \to & \cT \\

 & \qd & \mapsto & [h_\qd]
\end{array}
\]
which has been proved to be a homeomorphism \cite{wolf:teichmuller}. %% [Is the reference correct??]
This is called the Wolf parameterization centered at $[J_0]$.

The differential of $\wolf_{J_0}$ at $0$ can be easily computed:
\[
  d_0(\wolf_{J_0})(\qd)=[\nu_\qd]~.
\]

\subsubsection{Minimal Lagrangian maps}

Given two hyperbolic metrics $h, h^\dual$ on $S$, a map
$m:(S, h)\rightarrow (S, h^\dual)$ is minimal Lagrangian if
it is area-preserving and its graph is a minimal surface in
$(S\times S, h\oplus h^\dual)$.

A map $m:(S, h)\rightarrow (S, h^\dual)$ is a minimal Lagrangian map iff 
there exists an operator $b:TS\rightarrow TS$ such that
\begin{itemize}
\item $b$ is positive and self-adjoint;
\item $\det b=1$;
\item $b$ is solution of Codazzi equation $d^\nabla b=0$ for the Levi-Civita connection $\nabla$ of $h$;
\item $m^*(h^\dual)=h(b\bullet, b\bullet)$
\end{itemize}

Labourie \cite{L5} and Schoen \cite{schoen:role} proved that
there exists a unique such minimal Lagrangian map $m:(S, h)\rightarrow (S, h^\dual)$
isotopic to the identity.
In particular the Labourie operator of the pair $(h, h^\dual)$ is the
operator $b$ as above, such that the metric $h(b\bullet, b\bullet)$ is isotopic to
$h^\dual$.
We will say that the pair is normalized if $h^\dual=h(b\bullet, b\bullet)$, or equivalently
if the identity $\Id:(S, h)\rightarrow (S, h^\dual)$ is a minimal Lagrangian map.

%% It is well known that if 
If $(h, h^\dual)$ is a pair of normalized hyperbolic metrics, then the
metrics
\[
   h_c=h((\En+b)\bullet, (\En+b)\bullet)\,,\qquad h'_c=h(b\bullet, \bullet)
\]
are conformal, because $(\En+b)^2 = (2+\tr(b))b$ since $\det(b)=1$. %% [added this mini-proof]

So they %% determines 
determine a conformal structure, denoted here by $c$.
The identity maps
\[
    (S, c)\rightarrow(S, h)\,,\qquad (S, c)\rightarrow(S, h^\dual)
\]
are harmonic maps with opposite Hopf differential.

The conformal structure is called the {\it center} %% [added italics]
of the pair $(h, h^\dual)$.

\begin{proof}[Proof of Proposition~\ref{prop:energy}]
Real-analyticity of the energy functional $\Ene(\bullet,h):\cT\rar\R$
follows mimicking the arguments of \cite{eells-lemaire:deformation}.
Properness was proven by Tromba \cite{fischer-tromba:cell}.

If $\hat{c}$ is a metric in the conformal class of $c$ and
$h=\qd+e\hat{c}+\ol{\qd}$, then $h^\dual=-\qd+e\hat{c}-\ol{\qd}$. Thus,
$h((\En+b^2)\bullet,\bullet)=h+h^\dual=2e\hat{c}$, which implies
\[
2\Ene(c,h)=\int_S 2e\da_{\hat{c}}=2\int_S \sqrt{\det(\En+b^2)}\da_h
=2\int_S\tr(b)\da_h=2F(h,h^\dual)
\]
Finally, fix $h$ and let $(c_n)$ be a divergent sequence in $\cT$.
This determines a sequence $(h^\dual_n)$ such that
$c_n$ is the center of $(h,h^\dual_n)$. This $(h_n^\dual)$ is divergent too, for otherwise a subsequence of
$F(h,h^\dual_n)$ and so of $\Ene(c_n,h)$ would remain bounded,
contradicting the properness of $\Ene(\bullet,h)$.
Again up to subsequences, we can then assume that
$\theta_n h^\dual_n$ converge to a nonzero measured lamination
$\lambda$, where $\theta_n$ is a positive sequence that converges to zero.
It follows from Proposition~6.15 of \cite{cyclic}
that $\theta_n F(h,h_n^\dual)\rar \ell_\lambda(h)>0$
and so $F(h,h_n^\dual)\rar +\infty$.
\end{proof}

\subsection{The $\AdS^3$ space}

The 3-dimensional anti-de Sitter space $\AdS^3$ can be defined much like the hyperbolic space. 
Let $\R^{2,2}$ denote $\R^4$ with the symmetric bilinear form 
$\langle\bullet,\bullet\rangle_{(2,2)}$of signature $(2,2)$. Then
$$ \AdS^3 := \{ x\in \R^{2,2}~|~ \langle x,x\rangle_{(2,2)} =-1\}~, $$
with the induced metric. 

We refer the reader to \cite{mess,mess-notes} for the main properties of $\AdS^3$, and only recall
here a few key properties, without proof. In many respects $\AdS^3$ is reminiscent of $\Hyp^3$, while
some of its properties, in particular concerning its isometry group, also relates to the 3-dimensional
sphere $S^3$. 

The space $\AdS^3$ is Lorentzian, with constant curvature $-1$. It is not simply connected, however,
and its fundamental group is infinite cyclic. As a Lorentz space, it has three types of geodesics: space-like
and light-like %% [added this]
geodesics are open lines, while time-like geodesics are closed, of length $2\pi$. When considering 
$\AdS^3$ as a quadric in $\R^{2,2}$ as above, the geodesics in $\AdS^3$ are the intersections of
$\AdS^3$ with the 2-dimensional planes containing $0$ in $\R^{2,2}$. The space-like
totally geodesic planes in $\AdS^3$ are isometric to the hyperbolic plane. 

The multiplication by $\pm 1$ acts on every sphere in $\R^{2,2}$ and so on $\AdS^3$ and on the quadric $Q:=\{x\in\R^{2,2}\,|\,\langle x,x\rangle_{(2,2)}=0\}$.
The projective model of $\AdS^3$ identifies $\AdS^3/\{\pm 1\}$ to 
one connected component of the complement in $\R\Pr^3\setminus\PP Q$. 
The space-like geodesics in $\AdS^3/\{\pm 1\}$ then correspond to the projective lines
intersecting the boundary quadric $\PP Q$ in two points, while light-like geodesics are tangent to $\PP Q$, and time-like geodesics do not intersect $\PP Q$. Taking the double cover of
$\R\Pr^3$ (namely taking the quotient of $\R^{2,2}$ by $\R_+$) yields a projective model of $\AdS^3$ inside $S^3$. 
This projective model is one way to define the boundary at infinity of $\AdS^3$. It is topologically
a torus, with a Lorentz conformal structure.

The isometry group of $\AdS^3$ is $\mathrm{O}(2,2)$. However, up to finite quotient,
%index,  FB: I guess that PSL_2xPSL_2 is not a finite index subgroup of O(2,2) 
this isometry group splits
as the product of two copies of $\PSL(2,\R)$, more precisely its identity component is isomorphic to 
$(\SL(2,\R)\times \SL(2,\R))/\{\pm 1\}$. 

\subsection{Space-like surfaces in $\AdS^3$} \label{ssc:ads}

The local theory of space-like surfaces %surface 
in $\AdS^3$ is very similar to that of surfaces in 
the Euclidean or in the hyperbolic 3-dimensional space. Here again we only briefly recall without
proof some basic facts which will be useful below (see \cite{} for a treatment of this subject).

Let $\St\subset \AdS^3$ be a space-like surface,
and let $\nor$ be a unit normal vector field on $\St$. Given such a surface, we will call $I$ its induced
metric (or first fundamental form). The shape operator of $\St$ is an $I$-self-adjoint bundle morphism $B:T\St\to T\St$
defined by 
$$ \forall x\in \St\ \forall v\in T_x\St\quad Bv=-D_v\nor~, $$
where $D$ is the restriction to $\St$ of the Levi-Civita connection of $\AdS^3$. 
%The operator $B$ self-adjoint for $I$. 

The shape operator satisfies two basic equations. 
\begin{itemize}
\item The Codazzi equation: if $\nabla$ is the Levi-Civita connection of $I$ and $u,v$ are two
vector fields on $\St$, then $d^\nabla B=0$.
%$$ D_u(Bv)-D_v(Bu)-B([u,v])=0~. $$
%This is often written in a more compact form as $d^DB=0$, which is understandable if one
%considers $B$ as a 1-form on $S$ with values in the tangent bundle and considers $d^D$ as
%the exterior differential with respect to the connection $D$ on the tangent bundle.
%FB: I put the definition of Codazzi equation in subsection 2.1.4
\item The Gauss equation: the curvature of $I$ on $\St$ is equal to $K=-1-\det(B)$. 
\end{itemize}

The second and third fundamental forms of $\St$ are then defined by 
$$ \forall x\in \St\ \forall u,v\in T_x \St\quad \II(u,v)=I(Bu,v),\quad \III(u,v)=I(Bu,Bv)~. $$

\subsection{Globally hyperbolic $\AdS^3$ manifolds}

Let $\mgh$ be a Lorentz 3-dimensional manifold locally modeled on $\AdS^3$. We say that $\mgh$ is 
{\it maximal globally hyperbolic}, or MGH, if:
\begin{itemize}
\item it contains a closed space-like surface ({\it Cauchy surface}),
\item any inextendible time-like curve intersects the Cauchy surface exactly once,
\item it is maximal (under inclusion) among $\AdS^3$ manifolds having those properties, that is,
if $\mgh'$ is another 3-dimensional $\AdS$ manifold having the previous two properties and $i:\mgh\rightarrow \mgh'$ 
is an isometric embedding, then $i$ is onto.
\end{itemize}

Mess realized that some properties of MGH $\AdS^3$ manifolds are remarkably close of those of 
quasifuchsian hyperbolic manifolds. Among the analogies are the following points, which will be
useful below.

\begin{itemize}
\item The space of MGH $\AdS^3$ metrics on a fixed manifold $S\times \R$ is parameterized by the product
of two copies of $\cT$, the Teichm\"uller space of $S$, as with the Bers double uniformization theorem
for quasifuchsian manifolds. In the $\AdS^3$ case this parameterization comes from the holonomy representation
%%$\rho$ of an MGH AdS structure, which takes values in $\isom_0($\AdS^3$)=PSL(2,\R)\times \PSL(2,\R)$, and
$\rho$ of an MGH AdS structure, which takes values in 
$\isom_0(\AdS^3/\{\pm 1\})=\PSL(2,\R)\times \PSL(2,\R)$, and
%$\isom_0(\AdS^3)=PSL(2,\R)\times \PSL(2,\R)$, and  
therefore splits as two representations $\rho_l,\rho_r$ in $\PSL(2,\R)$. Mess \cite{mess} proves that
those two representations have maximal Euler class, so that they are holonomy representations of
two hyperbolic metrics $h_l, h_r$ on $S$. Any pair $(h_l, h_r)\in  \cT\times \cT$ can be obtained
from a unique MGH $\AdS^3$ structure.
\item A MGH $\AdS^3$ manifold $\mgh$ contains a smallest closed non-empty convex subset, called its convex
core $C(\mgh)$. (A subset $C$ of $\mgh$ is convex if any geodesic segment with endpoints in $C$ is
contained in $C$.) The boundary of $C(\mgh)$ is the disjoint union of two spacelike pleated surfaces, 
except in the ``Fuchsian'' case where $C(\mgh)$ is a totally geodesic surface. Each of those pleated
surfaces has an induced metric which is hyperbolic, and its pleating is encoded by a measured
lamination, as in the quasifuchsian setting. 
\item The complement of $C(\mgh)$ has a unique foliation by convex, space-like surfaces, with constant 
curvature varying monotonically from $-1$ (near the convex core) to $-\infty$ (near the 
initial/final singularity) on each side of the convex core, see \cite{BBZ2}. This is similar
to what happens for quasifuchsian manifolds or more generally hyperbolic ends, see \cite{L5}.
\end{itemize}

\subsection{The duality between convex surfaces in $\AdS^3$ manifolds}

There is a well-known ``projective'' duality (or polarity) between points and hyperplanes in the projective space, or
in the sphere. This duality has a hyperbolic version, which associates to a point in $\Hyp^n$ a space-like hyperplane
in the de Sitter space $\dS^n$, and to an oriented hyperplane in $\Hyp^n$ a point in $\dS^n$, see \cite{HR}. 

A similar duality exists between points and hyperplanes in $\AdS^3$ (or more generally in $\AdS^n$). 
We recall here its definition and its main properties. 
Consider $\AdS^3$ as a quadric in $\R^{2,2}$ as in Section \ref{ssc:ads}.
Every point $x\in \AdS^3$ is the intersection in $\R^{2,2}$
of $\AdS^3$ with a half-line $d$ starting from $0$
on which the bilinear form is negative definite. 
We call $d^\perp$ the oriented hyperplane orthogonal to $d$ in $\R^{2,2}$, so that the induced metric on $d^\perp$
has signature $(1,2)$. The intersection between $d^\perp$ and $\AdS^3$ is the disjoint union of two totally geodesic
space-like planes, one at distance $\pi/2$ in the future of $x$, the other at distance $\pi/2$ in the past of $x$.
We define the dual $x^\dual$ of $x$ as the oriented space-like plane which is the intersection of
$\AdS^3$ with $d^\perp$ at distance $\pi/2$ in the future of $x$.

Conversely, every totally geodesic space-like plane $P$ in $\AdS^3$ is the intersection of
$\AdS^3$ with a hyperplane $H$ of signature $(1,2)$ in $\R^{2,2}$. The orthogonal $H^\perp$ of $H$ then
intersects $\AdS^3$ in two antipodal points, and we define
the dual $P^\dual$ of $P$ as the intersection which is at distance $\pi/2$ in the past of $P$.

This duality relation has a number of useful properties. It is an involution, and the union of the planes dual 
to the points of a plane $P$ is the antipodal of the point dual to $P$. 

Consider now a smooth, space-like, strictly convex surface $\St$ in $\AdS^3$. Denote by $\St^\dual$ the set of points 
which are duals of the support planes of $\St$. The relation between $\St$ and $\St^\dual$ is based on the following
lemma (see e.g. \cite{these} for the analogous statement concerning the duality between $\HH^3$ and $\dS^3$,
the proof is the same in the AdS case).
%% [do you know of any better reference for this duality on smooth surfaces, in particular in the AdS case??]

\begin{lemma} \label{lm:dual}
Let $\St\subset \AdS^3$ be a smooth, space-like, locally strictly convex surface of constant curvature $K\in (-\infty,0)$. 
Then:
\begin{enumerate}
\item The dual of $\St$ is a smooth, locally strictly convex surface $\St^\dual$. 
\item The pull-back of the induced
metric on $\St^\dual$ through the duality map is the third fundamental form of $\St$, and vice versa.
\item If $\St$ is a space-like surface of constant curvature $K\in (-\infty, -1)$ in $\AdS^3$ then its dual $\St^\dual$
is a space-like surface of constant curvature $K^\dual=-K/(K+1)$.
\end{enumerate}
\end{lemma}

Consider now a smooth, space-like strictly convex surface $S$ in a MGH $\AdS^3$ manifold $\mgh$. The lift
of $S$ to the universal cover of $\mgh$ can be identified with a surface $\St$ in $\AdS^3$, invariant
under an action $\rho:\pi_1 S\to \isom(\AdS^3)$. The dual surface $\St^\dual$ is then also invariant
under $\rho$ so that it corresponds to a surface $S^\dual$ in $\mgh$.

\begin{lemma} \label{lm:GHdual}
If $\mgh$ is a MGH $\AdS^3$ manifold and $S$ is a space-like, past-convex surface in the future of the
convex core of $\mgh$, then its dual $S^\dual$ is space-like, future-convex surface in the past of 
the convex core. 
\end{lemma}

\subsection{Constant curvature foliations in MGH $\AdS^3$ manifolds}

An important fact used below is the existence of a foliation of the complement of the convex
core of a MGH $\AdS^3$ manifold by surfaces of constant curvature. This is described in the following
result obtained by Barbot, B\'eguin and Zeghib.

\begin{theorem}[Barbot, B\'eguin, Zeghib \cite{BBZ2}] \label{tm:foliation}
Let $\mgh$ be a MGH $\AdS^3$ manifold, and let $K\in (-\infty, -1)$. There is a unique
past-convex (resp. future-convex) closed space-like surface in the future (resp. past)
of the convex core, with constant curvature $K$. Those surfaces form a foliation of
the complement in $\mgh$ of the convex core.
\end{theorem}

Similar statements hold in the de Sitter and the Minkowski case, see \cite{BBZ2}.

\subsection{The landslide flow}

The landslide flow can be defined in at least three related ways, %% [displaced "related"]
each of which
can be convenient in some cases:
\begin{itemize}
\item in terms of harmonic maps and holomorphic quadratic differentials,
\item using minimal Lagrangian maps,
\item in terms of 3-dimensional globally hyperbolic $\AdS$ manifolds.
\end{itemize}
We briefly recall here two definitions, one in terms of minimal Lagrangian maps,
the other in terms of 3-dimensional $\AdS$ manifolds.
More details can be found in \cite{cyclic}.

Let $h,h^\dual$ be two hyperbolic metrics on $S$. We have recalled above that there is a
unique minimal Lagrangian map $m:(S,h) \to (S,h^\dual)$ isotopic to the identity. 
This map can be decomposed as $m=f^\dual\circ f^{-1}$, where $f:(S,c)\to (S,h)$
and $f^\dual:(S,c)\to (S,h^\dual)$ are harmonic maps isotopic to the identity, for the conformal
structure $c$ on $S$, and $f$ and $f^\dual$ have opposite Hopf differentials $\qd_f$ and $\qd_{f^\dual}=-\qd_f$.
For each $e^i\theta\in S^1$, there is a unique hyperbolic metric $h_\theta$ on $S$ 
such that $e^{i\theta}\qd_f$ is the Hopf differential of the harmonic map isotopic to
the identity from $(S,c)$ to $(S, h_\theta)$, and a unique hyperbolic metric 
$h_\theta^\dual$ such that $-e^{i\theta}\qd_f$ is the Hopf differential of the harmonic map isotopic to
the identity from $(S,c)$ to $(S, h_\theta^\dual)$. Then $(h_\theta,h^\dual_\theta)=\Land_{e^{i\theta}}(h,h^\dual)$ 
is the image of $(h,h^\dual)$ by the landslide flow with parameter $e^{i\theta}$.

In terms of AdS %% removed the $ here
geometry, the definition is the following. Let again $h,h^\dual$ be hyperbolic
metrics on $S$ and let $e^{i\theta}\in S^1$. There is a unique equivariant embedding of 
$\St$ in $\AdS^3$ with induced metric $\cos^2(\theta/2)h$ and third fundamental form 
$\sin^2(\theta/2)h^\dual$. The corresponding representation $\rho:\pi_1S\to \isom(\AdS^3)$ is
the holonomy representation of a MGH $\AdS^3$ manifold $\mgh$ (so that $\mgh$ contains a space-like
surface isometric to $\cos^2(\theta/2)h$ with third fundamental form equal to 
$\sin^2(\theta/2)h^\dual$). Then $(h_\theta, h^\dual_\theta)$ are the left and right hyperbolic
metrics of $\mgh$. 

\subsection{Hyperbolic ends}

An example of a hyperbolic end is an end of a quasifuchsian hyperbolic manifold, that is,
a connected component of the complement of the convex core in a quasifuchsian manifold.
More generally, a hyperbolic end $\hend$ is a 3-dimensional manifold homeomorphic
to $S\times \R_{>0}$, with a non-complete hyperbolic metric $g$ such that:
\begin{itemize}
\item $g$ is complete on the end of $S\times \R_{>0}$ corresponding to infinity,
\item $(\hend,g)$ has a metric completion for which the boundary corresponding
to $S\times \{ 0\}$ is a concave pleated surface.
\end{itemize}

Hyperbolic ends will appear in relation to the smooth grafting map, which is to landslides
as the grafting map is to earthquakes. 

A key result that we will use is that any hyperbolic end has a unique foliation by constant
curvature surfaces, with the curvature varying between $-1$ (close to the pleated surface
boundary) to $0$ (near the complete boundary), see \cite{L5}.

Given a hyperbolic end $\hend$, its boundary at infinity is the connected component of its 
boundary corresponding to $S\times \{ \infty\}$. This boundary at infinity $\partial_\infty \hend$
is the quotient of a domain in $\C\Pr^1=\partial_\infty \Hyp^3$ by an action of $\pi_1S$ by
complex projective transformations, so that $\partial_\infty \hend$ is endowed with a 
complex projective structure.

\subsection{The smooth grafting maps} \label{ssc:smooth-grafting}

We will consider two versions of smooth grafting:
\begin{itemize}
\item the map $SGr':\R_{>0}\times\cT\times \cT \to \cCP$ sending two hyperbolic metrics and
a real parameter to a complex projective structure on $S$,
\item the map  $sgr':\R_{>0}\times\cT\times \cT\to \cT$ which is the composition of 
$SGr'$ with the projection from $\cCP$ to $\cT$ sending a complex projective structure
to the underlying complex structure.
\end{itemize}

The map $SGr'$ can be defined using 3-dimensional geometry as follows. Let again $h,h^\dual\in \cT$
be two hyperbolic metrics on $S$, and let $s>0$. 
Up to global isometry there is a unique equivariant convex embedding
of $\St$ in $\Hyp^3$ such that the induced metric is $\cosh^2(s/2)\tilde{h}$ and the third fundamental form is 
$\sinh^2(s/2)\tilde{h}^\dual$. The corresponding representation is the holonomy representation of a 
hyperbolic end $\hend$ (so that $\hend$ contains a convex surface with induced metric $\cosh^2(s/2)h$ and third fundamental form
$\sinh^2(s/2)h^\dual$). We define $SGr'_s(h,h^\dual)$ as the complex projective structure on $\partial_\infty \hend$.

 %background

\section{First order computations}\label{sec:first}
%% [added "s" at computations]

Let us fix a normalized pair of hyperbolic metrics  $(h, h^\dual)$ and denote by $J$, $J^\dual$
the corresponding complex structures. Let us fix also a holomorphic quadratic  differential
$\qd\in\mathcal Q(J)$ and let us consider the family of hyperbolic metrics
\begin{equation}\label{eq:wolf}
   h_t= t\qd+e(t)h+t\bar\qd
\end{equation}
given by (\ref{eq:harmo}).

We will denote by $\alpha_t$ the positive self-adjoint operator such that
\[
%%  h_t=h(\alpha_t, \alpha_t)~.
  h_t=h(\alpha_t\bullet, \alpha_t\bullet)~.
\]
Notice that %% we have that 
\begin{equation}\label{eq:alpha}
\alpha_t^2=e(t)\En+2t\nu_{\qd}~.
\end{equation}

Moreover we will denote by $b_t$ the Labourie operator of the pair $(h_t, h^\dual)$.
Let us stress that in general $(h_t, h^\dual)$ is not a normalized pair of hyperbolic metrics,
so  $h_t^\dual=h_t(b_t\bullet, b_t\bullet)$ does not coincide with $h^\dual$, but there
is a continuous family of diffeomorphisms $m_t:S\rightarrow S$ such that
\begin{equation}\label{eq:hypisotopic}
  h^\dual_t=m_t^*(h^\dual)~.
\end{equation}

In this section we will point out some relations between
the first variation $\alpha_t$ and the first variation of $b_t$.
This technical computation will be the key tool
to prove Theorems \ref{tm:hamiltonian} and \ref{tm:cvx}.

\subsection{First order variation of $\alpha_t$}

Denote by $\dot\alpha$ the derivative  of $\alpha_t$ at $t=0$.
By (\ref{eq:alpha}) we simply see that
\begin{equation}\label{eq:dota0}
   \dot\alpha=\nu_{\qd}~.
\end{equation}
In particular we deduce that $\dot\alpha$ is a self-adjoint operator such that
\begin{equation}\label{eq:dota}
\tr\dot\alpha=0~,\qquad  d^{\nabla}\dot\alpha=0~.
\end{equation}

\begin{remark}
The computation we  will make in this section only depend on the properties
above of $\dot\alpha$. That is, 
the %% added
results of this section
are valid for any family of hyperbolic metrics  $h_t=h(\alpha'_t, \alpha'_t)$ 
supposing that $\dot\alpha'$ 
verifies %% verify 
(\ref{eq:dota}).

On the other hand, in order to compute the Hessian of $F$ with respect to
the Weil-Petersson metric, it will be necessary to use the deformation given by (\ref{eq:wolf}).
\end{remark}

\subsection{First order variation of the area form of $h_t$}

Let $\da_t=\da_{h_t}$ denote the area form of $h_t$ and let $\frac{d}{dt}(\da_t)$ 
be %% added
its time-derivative.
Since we have
\[
   \da_t=\det(\alpha_t)\da_h~,
\]
we easily deduce that
\begin{lemma} \label{lm:omegadot}
$ \frac{d}{dt}(\da_t)=\tr(\alpha_t^{-1}\dot\alpha_t)\da_t$.
\end{lemma}

Notice in particular that $\frac{d}{dt}(\da_t)|_{t=0}=0$.

\subsection{First order variation of $b$}

In order to get information about $\dot b$ at $t=0$ we will differentiate the identities satisfied
by $b_t$. In particular we have
\begin{itemize}
\item $\det b_t=1$,
\item $b_t$ is $h_t$-self-adjoint,
\item $d^{\nabla^t}b_t=0$, where $\nabla^t$ is the Levi-Civita connection for $h_t$,
\item $h_t(b_t\bullet, b_t\bullet)=m_t^*(h^\dual)$, where $m_t$ is a smooth family
of diffeomorphisms of $S$ such that $m_0=\En$.
\end{itemize}

Differentiating the first identity, we get
\begin{equation}\label{eq:dotb1}
  \tr(b^{-1}\dot b)=0~.
\end{equation}

About the second property, notice that the fact that $b_t$ is $h_t$-self-adjoint is equivalent to
requiring %% instead of require, I think this is the grammatically correct version in principle 
that $\tr (J_tb_t)=0$, where $J_t$ is the complex structure compatible with $h_t$.
Differentiating this identity and using that $J_t=\alpha_t^{-1}J\alpha_t$, and that
$J\dot\alpha=-\dot\alpha J$ one finds:

\begin{lemma}\label{lm:sa}
$\displaystyle\tr (J\dot b)=-2\tr (J\dot\alpha b)$.
\end{lemma}

In order to get the infinitesimal information by the last properties
it is convenient to introduce the operator $\psi_t=\alpha_tb_t$. Notice that we have
\begin{equation}\label{eq:psi}
   h^\dual_t=h(\psi_t, \psi_t)~,
\end{equation}
and the infinitesimal deformation of $\psi$ at $t=0$ is simply $\dot\psi=\dot\alpha b+\dot b$.

By the fact that $h^\dual_t$ is a trivial family of hyperbolic metrics
 we deduce the following relation.

\begin{lemma}\label{lm:triv}
There exist a family of vector fields $Z_t$  on $S$ and a family of
functions $\skew_t$ on $S$ such that
\begin{equation}\label{eq:X}
\nabla^{\dual,t} Z_t +\skew_t J^\dual_t=\psi_t^{-1}\dot\psi_t
\end{equation}
where $J^\dual_t$ is the complex structure for $h^\dual_t$ and $\nabla^{\dual,t}$
is the Levi-Civita connection of $h_t^\dual$.
\end{lemma}

\begin{proof}
By (\ref{eq:psi}),
\begin{align*}
\dot{h}_t^\dual=h(\dot\psi_t\bullet, \psi_t\bullet)+h(\psi_t\bullet,\dot\psi_t)
=h^\dual(\psi^{-1}_t\dot\psi_t\bullet, \bullet)+h^\dual(\bullet, \psi^{-1}_t\dot\psi_t\bullet)~.
\end{align*}
On the other hand, let us consider the field 
%%$$ Z_t(p)=d_p(m_t^{-1})\left(\frac{\partial m_t(p)}{\partial t}(t)\right)~. $$
$$ Z_t(p)=d_p(m_t^{-1})\left(\frac{\partial m_t(p)}{\partial t}\right)~. $$
By (\ref{eq:hypisotopic}), we have
\[
\dot{h}_t^\dual=h_t^\dual(\nabla^{\dual,t} Z_t,\bullet)+h^\dual_t(\bullet,\nabla^{\dual,t} Z_t)
\]
from which we obtain that the difference $\nabla^{\dual,t}Z_t-\psi_t^{-1}\dot\psi_t$
is $h_t^\dual$-skew-symmetric, and the conclusion follows.
\end{proof}

Notice that at $t=0$ we have $\psi_0=b$ so we deduce 
that %% added
$\nabla^\dual Z_0+\skew_0 J^\dual=b^{-1}\dot{\psi}$.
It can be shown that $\nabla^\dual=b^{-1}\nabla b$ (see \cite[Lemma 3.3]{cyclic} in the case $\theta=\pi$), 
whereas $J^\dual=b^{-1}Jb$. So we can rewrite the identity above in the form
\begin{equation}\label{eq:X2}
  \dot\psi=\nabla Y+\skew Jb~,
\end{equation}
where we have put $Y=bZ_0$ and $\skew=\skew_0$.

Finally differentiating the identity $d^{\nabla^t}b_t=0$ at $t=0$,
we get:

\begin{lemma}\label{lm:cdpsi}
$d^{\nabla}(\dot \psi)=0$.
\end{lemma}

The proof of Lemma~\ref{lm:cdpsi} given below is based on the computation of $d^{\nabla^t}$, which relies on the following
two results.

\begin{lemma}\label{lm:cod}
There exists a family of vector fields $V_t$ on $S$ such that for $v,w\in TS$ we have
$(d^{\nabla}\alpha_t)(v,w)=\da_h(v,w) V_t$ with  $V_0=\dot V_0=0$
\end{lemma}

\begin{proof}
On a point $p\in S$ take any $h$-orthonormal basis $e_1, e_2$ of $T_pS$.
Then putting $V_t(p)=(d^{\nabla}\alpha_t)(e_1, e_2)$, it follows that
\[
 (d^{\nabla}\alpha_t)(v,w)=\da_h(v,w) V_t(p)
\]
for every $v,w\in T_p S$. Clearly $V_t$ smoothly depends on $p$ and $t$.
Since $\alpha_0=\En$, $V_0$ vanishes everywhere.
On the other hand, by the linearity of $d^\nabla$ we have
\[
   d^{\nabla}\dot\alpha=\da_h\otimes\dot V_0~,
\] 
and  by (\ref{eq:dota}) we deduce that $\dot V_0$ vanishes everywhere.
\end{proof}

\begin{lemma}\label{lm:lc}
Let $\nabla^t$ the Levi-Civita connection of the metric  $h_t$.
If %% Then, if 
$v,w$ are vector fields on $S$ we have
\[
   \nabla^t_vw=\alpha_t^{-1}\nabla^t_v(\alpha w) + h_t(W_t, v) J_t(w)
\]
where $J_t$ is the complex structure compatible with $h_t$,
$W_t=\det(\alpha_t^{-1})\alpha_t^{-1}V_t$, and 
$V_t$ is the field defined in Lemma \ref{lm:cod}.
\end{lemma}

\begin{proof}
Notice that the connection 
%% $\alpha_t^{-1}\nabla\alpha_t$ 
$\alpha_t^{-1}\nabla (\alpha_t\bullet)$ 
is compatible with the metric $h_t$ but is not symmetric (since we are not assuming that
$\alpha_t$ is a solution of the Codazzi equation for $h$).

In particular, the difference
\[
   T(v,w)=\nabla^t_v w-\alpha_t^{-1}\nabla_v(\alpha_t w)
\]
is a vector-valued $2$-form such that $T(v,\bullet)$ is $h_t$-skew-symmetric.
That is, there exists a 1-form $\zeta$ such that  $T(v,w)=\zeta(v)J_tw$.
On the other hand an explicit computation shows that
\[
 T(v,w)-T(w,v)=-\alpha_t^{-1}(d^{\nabla}\alpha_t)(v,w)=-\da_h(v,w)\alpha_t^{-1}V_t~.
\]
If 
$(e_1,e_2)$ %% $e_1, e_2$ 
is a positive  $h_t$-orthonormal basis, we have
$\zeta(e_1)J_t(e_2)-\zeta(e_2)J_t(e_1)= -\da_h(e_1, e_2)\alpha_t^{-1}V_t$, that is
\[
    \zeta(e_1)e_1+\zeta(e_2)e_2=W_t~,
\]
so $\zeta(e_i)=h_t(e_i, W_t)$ and the result follows.
\end{proof}

\begin{proof}[Proof of Lemma \ref{lm:cdpsi}]
Take two vector fields $v,w$ on $S$, and consider %% the 
the identity
\[
   \nabla^t_v b_t(w) -\nabla_w^t b_t(v)- b_t([v,w])=0.
\]
By Lemma \ref{lm:lc}, we can rewrite this identity as
\begin{eqnarray*} %% I've put an eqnarray to improve readability
    0 & = & \alpha_t^{-1}(\nabla_v(\alpha_tb_t(w))-\nabla_w(\alpha_tb_t(v))-\alpha_tb_t([v,w]))
   +h_t(W_t, v)J_t(w)-h_t(W_t, w)J_t(v) \\
   & = & \alpha_t^{-1} d^{\nabla}(\alpha_t b_t)(v,w)
    +h_t(W_t, v)J_t(w)-h_t(W_t, w)J_t(v)
\end{eqnarray*}
Since $V_0=\dot V_0=d^\nabla\psi_0=0$ we have $W_0=\dot W_0=0$, so 
differentiating the last identity at $t=0$   we have
\[
     d^{\nabla}(\dot\psi)=0~.
\]
\end{proof}

 Lemma \ref{lm:cdpsi} implies the following interesting relation between
the field $Y$ and the function $\skew$ appearing in (\ref{eq:X2}).

\begin{lemma}\label{lm:grad}
We have
\[
  JY=b^{-1} \grad\skew~.
\]
\end{lemma}

\begin{proof}
Take a positive $h$-orthonormal basis 
$(e_1,e_2)$ %% $e_1, e_2$ 
of $T_pS$.
By Lemma \ref{lm:cdpsi} we have that $d^\nabla(\nabla Y+\skew Jb)=0$.
It follows that
\[
  d^\nabla(\nabla Y)(e_1, e_2)+ d^{\nabla}(\skew Jb)(e_1, e_2)=0~.
\]

We have $(d^\nabla\nabla Y)(e_1, e_2)=R(e_1, e_2)Y=JY$.
On the other hand, since $d^\nabla (Jb)=Jd^\nabla b=0$,
\[
  d^{\nabla}(\skew Jb)(e_1, e_2)=d\skew(e_1)Jb e_2-d\skew(e_2)Jb e_1=
  JbJ(d\skew(e_1)e_1+d\skew(e_2)e_2)=JbJ\grad\skew=-b^{-1}\grad\skew~,
\]
and the conclusion follows.
\end{proof}

\subsection{The function $F$ and its variation}

We consider on $\mathcal M_{-1}\times\mathcal M_{-1}$ the function
\[
   \tilde F(h, h^\dual)=\int\tr(b)\da_h~,
\]
where $b$ is the Labourie operator of the pair $(h, h^\dual)$.
Clearly $F$ is invariant by the action of $\Diffeo_0\times \Diffeo_0$,
so it induces a smooth function
\[
   F:\mathcal T\times\mathcal T\rightarrow\R~.
\]
In this section we will compute the derivative of the function $F_t:=F([h_t], [h^\dual])$ with respect to $t$.

\begin{prop}\label{prop:1der}
The first-order derivative of $F_t$ is
\begin{equation}\label{eq:1der}
\dot{F}_t=\int_S \left[\tr(b_t)\tr(\alpha_t^{-1}\dot{\alpha}_t)-\tr(\alpha_t^{-1}\dot{\alpha}_t b_t) \right] \da_t~.
\end{equation}
In particular, at $t=0$,
\[
\dot{F}=-\int_S \tr(\dot{\alpha}b)\da_h~.
\]
\end{prop}

\begin{proof}[Proof of Proposition~\ref{prop:1der}]
We want to compute
\begin{equation}\label{def:first}
\dot{F}_t=\int_S \tr(\dot{b}_t)\da_t+\tr(b_t)
{\textstyle{\frac{d}{dt}}}(\da_t)~.
\end{equation}

%% after this point I changed all X_t to Z_t.
Since $\psi_t=\alpha_t b_t$, Equation (\ref{eq:X}) can be rearranged as
\[
%%\dot\psi_t=\alpha_t\nabla^t (b_t X_t)-\skew_t \alpha_t J_t b_t~.
\dot\psi_t=\alpha_t\nabla^t (b_t Z_t)+\skew_t \alpha_t J_t b_t~.
\]
In particular,
\[
\dot{b}_t=-\alpha_t^{-1}\dot{\alpha}_t b_t+\nabla^{t}(b_t Z_t) +\skew_t J_t b_t 
\]

%% I've added a few details to the following paragraph
%Direct computations show that $(J_tb_t)^2=-\En$, while $J_tb_t$ is anti-self-adjoint
%for the metric $h(b_t\bullet,\bullet)$, which is in the conformal class of the center $c_t$.
%It follows that $J_t b_t$ is the $J_{c_t}$-operator at $c_t$, and therefore

%% FB: I have simplified the argument above
Since $b_t$ is $h_t$-self-adjoint $\tr(J_tb_t)=0$, and so
$\tr(\skew_t J_t b_t)=0$.

Moreover, $\tr(\nabla^t (b_t Z_t))=\mathrm{div}_t(b_t Z_t)$ and so
\[
\int_S \tr(\nabla^t (b_t Z_t))\da_t=\int_S \mathrm{div}_t(b_t Z_t)\da_t=0~.
\]

But $\frac{d}{dt}(\da_t)=\tr(\alpha_t^{-1}\dot\alpha_t)\da_t$ by Lemma \ref{lm:omegadot}, so we obtain that
\begin{equation}\label{eq:pF-first}
\dot{F}_t=\int_S \left[\tr(b_t)\tr(\alpha_t^{-1}\dot{\alpha}_t)-\tr(\alpha_t^{-1}\dot{\alpha}_t b_t) \right] \da_t~.
\end{equation}
By (\ref{eq:dota}), at time $t=0$ we obtain
\begin{equation}\label{eq:F-first}
\dot{F}=-\int_S \tr(\dot{\alpha}b)\da_h~.
\end{equation}
\end{proof}

\begin{remark}\label{rk:gen}
Formula  (\ref{eq:1der}) holds for any family of deformations $h_t$ of the metric $h$,
even without assuming (\ref{eq:dota}). Notice indeed that both the proofs of  
Proposition \ref{prop:1der} and  Lemma \ref{lm:triv} do not make use of this hypothesis.
\end{remark}

 %first-order computations + hamiltonian (section 3 and 4)

\section{The landslide flow is Hamiltonian}

Let $F$ be the function on $\mathcal T\times\mathcal T$ defined in Section \ref{sec:first}.
The goal of this section is to show that $\frac{1}{4} F$ is the Hamiltonian
of the landslide flow with respect to the product symplectic form 
$\omega_{WP,1}+\omega_{WP,2}$.

If $(X,X^\dual)\in T_h\cT \oplus T_{h^\dual}\cT$ is the 
generator of the landslide flow at the point $(h, h^\dual)$, we need to prove that
it coincides with the symplectic gradient of $F$ at $([h],[h^\dual])$.

This is equivalent to showing that $X$ coincides with the symplectic gradient
of $F(\bullet, [h^\dual])$ at the point $h$ for $\omega_{WP}$, and analogously
that $X^\dual$ is the symplectic gradient of $F([h],\bullet)$ at $h^\dual$.

By a simple symmetry argument, it is sufficient to check the first point.
In particular, given any tangent vector $v\in T_{[h]}\mathcal T$ we need to show
that
\begin{equation}\label{eq:ham}
    \omega_{WP}(X, v)=\frac{1}{4}d(F(\bullet, [h^\star]))(v)~.
\end{equation}

Now, there exists a holomorphic quadratic differential $\qd$ such that
$v=[\nu_\qd]$.
Let $\dot J_\qd$ be  the first order variation associated with 
this %% the 
Beltrami differential $\nu_\qd$, and
let $\dot J_X\in T_{J}\mathcal A$ be the first order variation of the complex structure corresponding to
the landslide deformation of the metric $h_t=h(\beta_t\bullet, \beta_t\bullet)$ with 
%% $\beta_t=\cos t \En+\sin t Jb/2$. Notice that $d\pi(\dot J_\qd)=v$ and $d\pi(\dot J_X)=X$ so we
$\beta_t=\cos (t/2) \En+\sin (t/2) Jb$. Notice that $d\pi(\dot J_\qd)=v$ and $d\pi(\dot J_X)=X$ so
by (\ref{eq:ft2}) we have that
\[
\omega_{WP}(X, v)=\frac{1}{8}\int_S\tr(J\dot J^H_X\dot J^H_\qd)\da_h~.
\]
Now by (\ref{eq:belt}),  $\dot J_\qd=2J\nu_\qd=-2\nu_\qd J$,
whereas $\dot J_X=\frac{1}{2}(JJb-JbJ)$. In particular, since $\dot J_\qd$ is horizontal,
we get
\begin{equation}\label{eq:ham1}
\omega_{WP}(X,v)=\frac{1}{8}\int_S\tr(J\dot J_X\dot J_\qd)\da_h~=
-\frac{1}{4}\int_S\tr(b\nu_\qd)\da_h~.
\end{equation}

To compute the 
right-hand side %% right hand 
of (\ref{eq:ham}), we can consider the path of metrics
$h_t=t\qd+e(t)h+t\bar\qd$ as in (\ref{eq:wolf}). Then we have
\[
d(F(\bullet, [h^\dual]))(v)=\frac{d F([h_t], [h^\dual])}{dt}(0)~.
\]
With the 
notations %% notation 
of Section \ref{sec:first}, Proposition \ref{prop:1der} implies that
\[
d(F(\bullet, [h^\dual]))(v)=-\int_S \tr(\dot{\alpha}b)\da_h~.
\]
By (\ref{eq:dota0}), comparing this identity with (\ref{eq:ham1}), we get
(\ref{eq:ham}).

\section{Convexity of $F$}

The aim of this section is to show that the function 
$F:\mathcal T\times\mathcal T\rightarrow\R$ is convex on each factor
with respect to the Weil-Petersson metric.

In \cite{wolf:teichmuller}, it has been shown that the family of metrics $h_t$ introduced in
(\ref{eq:wolf}) determines a path in $\mathcal T$ which is $WP$-geodesic at $t=0$.
So fixing $h^\dual$, the Hessian of the function $F(\bullet, h^\dual)$ at $[h]$
is determined by
\[
\Hess (F(\bullet,[h^\dual]))([\nu_\qd], [\nu_\qd])=\frac{d^2 F_t}{dt^2}(0)~,
\]
where $F_t=F([h_t], [h^\dual])$.

Using the 
notations %% notation 
of Section \ref{sec:first},
by differentiating Equation (\ref{eq:pF-first}), we have
\begin{align*}
\ddot{F}_t & =\int_S \Big[
\tr(\dot{b}_t)\tr(\alpha_t^{-1}\dot{\alpha}_t)+
\tr(b_t)\tr[\alpha_t^{-1}\ddot{\alpha}_t-(\alpha_t^{-1}\dot{\alpha}_t)^2]
+\tr[(\alpha_t^{-1}\dot{\alpha}_t)^2-\alpha_t^{-1}\ddot{\alpha}_t) b_t]
-\tr(\alpha_t^{-1}\dot{\alpha}_t\dot{b}_t)\Big]\da_t + \\
& \qquad
+\int_S (\tr(b_t)\tr(\alpha_t^{-1}\dot\alpha_t)-\tr(\alpha_t^{-1}\dot\alpha_tb_t))
{\textstyle{\frac{d}{dt}}}(\da_t)~.
\end{align*}
We have already seen that $\frac{d}{dt}(\da_t)=0$ at time $t=0$. So, at time $t=0$, we obtain
\begin{equation}\label{eq:F-second}
\ddot{F}=\int_S \left[
\tr(b)\tr(\xi)-\tr(\xi b)-\tr(\dot\alpha\dot b)
\right]\da_h~,
\end{equation}
where $\xi:=\ddot{\alpha}-\dot{\alpha}^2$.

Since $\dot\alpha$ is self-adjoint and traceless we deduce that
$\dot\alpha^2$ is a non-negative multiple of the identity.
%\begin{equation}\label{eq:dot}
%\dot\alpha^2=\kappa \En~,
%\end{equation}
%where $\kappa=\tr(\dot\alpha^2)/2$.
On the other hand, by comparing the relation $h_t=h(\alpha_t, \alpha_t)$ with (\ref{eq:wolf}) we get
\[
\ddot e(0) h=\ddot h= 2h((\dot\alpha^2+\ddot\alpha)\bullet, \bullet)~.
\]
so we deduce that
\begin{equation}\label{eq:c1}
   \dot\alpha^2+\ddot\alpha=\frac{1}{2}\ddot e(0)\En~. 
\end{equation}

In \cite{wolf:teichmuller}, the function $e(t)$ has been computed up to
%% is the reference to \cite{wolf:teichmuller} correct ??
the second order. More precisely, if in local conformal coordinates 
$\qd=\ph(z)dz^2$ and $h=h(z)|dz|^2$, then we have
\begin{equation}  \label{eq:ddot1}
      e(t)=1+t^2\left(
    \frac{|\ph(z)|^2}{h(z)^2} +2(2-\Delta)^{-1}\frac{|\ph(z)|^2}{h(z)^2}
    \right )+O(t^3)~.
\end{equation}

With the real notation, $\displaystyle
\frac{|\ph(z)|^2}{h(z)^2}=|\nu_\qd|^2%=\dot\alpha^2
=\tr(\dot\alpha^2)/2 %=\kappa
$. So we can rewrite (\ref{eq:ddot1})
as %% added
\begin{equation}\label{eq:wolf2}
\frac{1}{2}\ddot e(0)=|\nu_\qd|^2+2(2-\Delta)^{-1}(|\nu_\qd|^2)~.
\end{equation}

Using %(\ref{eq:dot}) and
(\ref{eq:wolf2}) in (\ref{eq:c1}) we have
\[
   \ddot\alpha=2(2-\Delta)^{-1}(|\nu_\qd|^2)\En~.
\]
In particular the operator $\xi$ in (\ref{eq:F-second}) is
equal to $\xi=[2(2-\Delta)^{-1}(|\nu_\qd|^2)-|\nu_\qd|^2]\En$.
Since $\xi$ is a multiple of the identity, we have that 
\begin{equation}
  \label{eq:mb}
 \tr (b)\tr (\xi)-\tr(b\xi)=\tr (b \xi)=\tr (b(\ddot\alpha-\dot\alpha^2))~. 
\end{equation}

In order to conclude the computation we need to estimate the integral of the term
$\tr(\dot\alpha\dot b)$ appearing in (\ref{eq:F-second}).

\begin{prop}\label{prop:psi}
We have
\begin{equation}\label{eq:psi1}
\int_S\tr(\dot\alpha\dot b)\da_h=-\int_S\tr(\dot\alpha^2b)\da_h- 
\int_S h(\grad\skew, b^{-1}\grad\skew)\da_h\int_S\skew^2\tr (b)\da_h~.
\end{equation}
In particular
\begin{equation}\label{eq:psi2}
\int_S\tr(\dot\alpha\dot b) \da_h\leq-\int_S\tr(\dot\alpha^2 b)\da_h~.
\end{equation}
\end{prop}

\begin{proof}
Multiplying the identity (\ref{eq:X2})
%\begin{equation}\label{eq:id1}
%  \dot b+\dot\alpha b=\nabla Y+\skew Jb~.
%\end{equation}
by $\dot\alpha$ and taking the trace, we get
\begin{equation}\label{eq:id2}
\tr(\dot\alpha\dot b)=-\tr(\dot \alpha^2b)+\tr(\dot\alpha\nabla Y)+\tr(\skew\dot\alpha Jb)~.
\end{equation}
Since $\dot\alpha$ solves the Codazzi equation $d^\nabla\dot\alpha=0$ we have
 $\dot\alpha\nabla Y=\nabla(\dot\alpha Y)-\nabla_Y\dot\alpha$. Taking the trace and considering
that $\tr\dot\alpha=0$, it results that
\[
\tr(\dot\alpha\nabla Y)=\Div(\dot\alpha Y)~.
\]

On the other hand, multiplying the identity (\ref{eq:X2}) by $J$ we get
\[
  J\dot b+J\dot\alpha b=\nabla (JY)-\skew b~.
\]
Taking the trace and using Lemma \ref{lm:sa} we have $\tr(J\dot\alpha b)=-\Div(JY)+\skew\tr(b)$.
%% I removed the - sign in front of the previous equation.
Using again that $\dot\alpha J=-J\dot\alpha$ we get
\[
\tr(\dot\alpha J b)=\Div(JY)-\skew\tr(b)~.
\]
In particular replacing this identity in (\ref{eq:id2}) we obtain that
\[
\tr(\dot\alpha\dot b)=-\tr(\dot\alpha^2b)+\Div(\dot\alpha Y)+\skew \Div(JY)-\skew^2\tr b~.
\]
Integrating, it results that
\[
\int_S\tr(\dot\alpha\dot b)\da_h=-\int_S\tr(\dot\alpha^2b)\da_h-\int_S h(\grad\skew, JY)\da_h
-\int_S\skew^2\tr (b)\da_h~.
\]
By Lemma \ref{lm:grad}, $h(\grad\skew, JY)=h(\grad\skew, b^{-1}\grad\skew)>0$, so the result easily follows.
\end{proof}

%% I've made the following into a proof, to make the exposition clearer
\begin{proof}[Proof of Theorem \ref{tm:cvx}]
Using (\ref{eq:psi1}) in (\ref{eq:F-second}) and taking into account (\ref{eq:mb}) we get
\begin{align*}
\ddot F=\int_S[\tr(b(\ddot\alpha-\dot\alpha^2))+\tr (b\dot\alpha^2)] \da_h
+\int_S h(\grad\skew, b^{-1}\grad\skew)\da_h+\int_S\skew^2\tr (b)\da_h=\\ 
\int_S\tr (b\ddot\alpha)\da_h+\int_S h(\grad\skew, b^{-1}\grad\skew)\da_h+\int_S \skew^2\tr (b)\da_h~.
 \end{align*}

In particular, we deduce that
\[
\ddot F\geq\int_S\tr(b\ddot\alpha)\da_h=\int_S2(2-\Delta)^{-1}(|\nu_\qd|^2)\tr (b)\da_h~.
\]

Now, let $\qd\neq 0$ and put $u:=(2-\Delta)^{-1}(|\nu_\qd|^2)$. We have that
$2u-\Delta u=|\nu_\qd|^2$ and, by 
the %% added
maximum principle,
$u>0$ (since $|\nu_\qd|^2$ is positive).
Hence, $\ddot F$ is positive.

A simple case is where $h=h^\dual$. In this a case
it is not difficult to check that $\skew=0$ and $b=\En$.
Then we simply get
\[
\ddot F=\int_S 4u\,\da_h=\int_S2(|\nu_\qd|^2+\Delta u)\da_h=\int_S|\nu_\qd|^2\da_h=2g_{WP}(\qd, \qd)~.
\]
\end{proof}

 %convexity (section 5)

\section{Smooth grafting} \label{sc:sgr}

We now turn to the smooth grafting map, and to the proof of Theorem
\ref{tm:sgr}.

\begin{subsection}{Notation and hypotheses}

Fix a point $[h]$ in Teichm\"uller space, where $h$ is a hyperbolic
metric on $S$.  Fix also $s>0$ and consider a one-parameter family $t\mapsto b_t$ of
$h$-Labourie operators, that gives a family of hyperbolic metrics
$h_t^\dual=h(b_t\bullet,b_t\bullet)$.

The smooth grafting $sgr'_{s}(h,h^\dual_t)$ is represented by the
metric $h^\#_t=h(\beta_t,\beta_t)$, where
\[
\beta_t=\cosh(s/2)\En+\sinh(s/2)b_t\, .
\]

We will show that the map
\[
sgr'_{s}(h,\bullet):\Teich\longrightarrow\Teich
\]
is surjective.

More precisely we will show the following results:
\begin{itemize}
\item[(a)]
$sgr'_{s}(h,\bullet)$ is proper;
\item[(b)]
the degree of the map
$sgr'_{s}(h,\bullet):\Teich\rightarrow\Teich$ is $1$.
\end{itemize}

Surjectivity is an immediate consequence of $(b)$.
Notice that $(a)$ is needed to define the degree of the map
$sgr'_{s}(h,\bullet)$.
\end{subsection}

\begin{subsection}{The map $sgr'_{s}(h,\bullet)$ is proper} %% added "the map" so that title does not begin by symbol.
Let $[h_n^\dual]$ be a divergent sequence in $\Teich$ and
let $J^\#_n$ be the complex structure associated to
$h^\#_n$.
We want to show that $J^\#_n$ is a diverging sequence in $\Teich$.
The proof is based on the following lemma.

\begin{notation}
Let $f:S\rightarrow S$ be any smooth map.
For any two metrics $g,g'$ on $S$, denote by
$E(f;g,g')$ the energy of $f$ regarded
as a map $f:(S,g)\rightarrow (S,g')$.
%begin added
Since $E(f;g,g')$ is invariant by conformal
deformations of $g$, we sometimes
replace $g$ by its underlying conformal structure. 
%end added 
\end{notation}

\begin{lemma}\label{lm:energyest}
Let $h$ be a hyperbolic metric and let $s>0$. Given a $h$-Labourie operator $b$, call $\beta_c=\En+b$ and 
$\beta=\cosh(s/2)\En+\sinh(s/2) b$, 
% $\beta'=\En+\tau b$,
and consider the metrics $h_c=h(\beta_c\bullet, \beta_c\bullet)$ and
$h^\#=h(\beta\bullet,\beta\bullet)$. %% I've added the four \bullet
Then
\[
    E(f; h_{c}, h)< \tau^{-1}E(f; h^\#, h) 
%    E(f; h_{c}, h)< \tau^{-1}E(f; h', h) 
\]
where $\tau=\tanh(s/2)$, for any smooth map $f:S\rightarrow S$.
\end{lemma}
\begin{proof}
Notice that $h=h_{c}(\gamma, \gamma)$ with $\gamma=\beta_{c}^{-1}\beta$ and that
\begin{equation}\label{eq:est}
\da_{h_{c}}\leq
\tau^{-1}\cosh^{-2}(s/2)\da_{h^\#}
\end{equation}
as $\tau\det(\beta_c)=\tau(2+\tr(b))\leq
1+\tau^2+\tr(b)=\cosh^{-2}(s/2)\det(\beta)$.
%Notice that $h'=h_{c}(\gamma, \gamma)$ with $\gamma=\beta_{c}^{-1}\beta'$ and that
%\begin{equation}\label{eq:est}
%\da_{h_{c}}\leq
%\tau^{-1}\da_{h'}
%\end{equation}
%as $\tau\det(\beta_c)=\tau(2+\tr(b))\leq
%1+\tau^2+\tr(b)=\det(\beta')$.
%
%Moreover, as $\wti{h}'_n\leq h_{c_n}$, we have
%$e(f_n;\wti{h}'_n,h)\geq e(f_n;h_{c_n},h)$.

Now if $\alpha$ is the $h^\#$-self-adjoint operator such that
\[
     f^*(h)=h^\#(\alpha\bullet, \alpha\bullet)
\]
we have that
$f^*(h)=h_{c}(\gamma\alpha\bullet,\gamma\alpha\bullet)$,
and so
\[
   e(f; h^\#, h)=\frac{1}{2}\tr(\alpha^2)\,,\qquad\qquad e(f;
   h_{c},h)=\frac{1}{2}\tr((\gamma\alpha)^\dagger\gamma\alpha)~.
\]
where $^\dagger$ denotes the adjoint with respect to $h_c$.
Notice that the eigenvalues of $\gamma$ are less than
$\cosh(s/2)$
everywhere and so $h_c(\gamma v,\gamma v)<\cosh^2(s/2)h_c(v,v)$ for every nonzero tangent vector $v$.
Indeed if $k$ is the biggest eigenvalue of $b$, then the
eigenvalues of $\gamma$ are $\cosh(s/2)\frac{1+\tau k}{1+k}$ and
$\cosh(s/2)\frac{\tau+k}{1+k}$. It easily follows that
\[
  e(f; h_{c},h)<e(f;h^\#, h)\,.
\]
This inequality with (\ref{eq:est}) implies the statement.
\end{proof}

Now, let $\beta_{c_n}=\En+b_n$ and $h_{c_n}=h(\beta_{c_n}\bullet,\beta_{c_n}\bullet)$, and consider
the smoothly grafted metric
$h(\beta^\#_n\bullet,\beta^\#_n\bullet)$, where %% I've added this line to make it a metric
$\beta^\#_n=\cosh(s/2)\En+\sinh(s/2) b_n$,
defining the conformal class $J^\#_n$.

By Lemma \ref{lm:energyest} applied to the unique
harmonic map
$f_n:(S,J^\#_n)\rightarrow (S,h)$ isotopic to the identity,
we have
\[
E(\Id;c_n,h)\leq
E(f_n;c_n,h)
\leq
\tau^{-1} E(f_n;h^\#_n,h)~.
\]
So, as $h^\dual_n$ is diverging, $c_n$ is diverging too
and $E(\Id;c_n,h)\rightarrow\infty$. As a consequence,
$E(f_n;h_n^\#,h)\rightarrow\infty$, which implies that
the isotopy class of the underlying complex structure 
$[J^\#_n]$ is diverging in $\mathcal T$ (see \cite{wolf:teichmuller}).
\end{subsection}

\begin{subsection}{The degree of $sgr'_{s}(h,\bullet)$}
In this section we will compute the topological degree of
the map
\[
\cG:= sgr'_{s}(h,\bullet):\Teich\rightarrow\Teich
\]
and we will prove that it is equal to $1$.

In fact we will prove that $\cG^{-1}(h)=\{h\}$ and that the map
$\cG$ is locally invertible around $h$.

\begin{lemma}\label{lm:unique}
If $h^\#=sgr'_{s}(h,h^\dual)$ represents the same point in $\Teich$ as $h$,
then $h^\dual=h$.
\end{lemma}

This lemma is a simple consequence of  the following
statement. 

\begin{prop}\label{lm:energyest2}
Let $\hat h$ be the unique hyperbolic metric in the conformal class of $h^\#$.
Then $F([\hat h], [h^\dual])\leq F([h],[h^\dual])$ and the equality holds iff $h=h^\dual$.
\end{prop}
\begin{proof}
By \cite{cyclic} we have  that
\[
  F([\hat h], [h^\dual])=\inf_{[h']\in\mathcal T}E([h'],[\hat{h}])+E([h'], [h^\dual])\leq
E([c],[\hat h])+E([c], [h^\dual])
\]
so we only need to show that $E([c],[\hat h])\leq E([c], [h])=E([c],[h^\dual])$ and 
that %% added
the equality 
holds only if $h=h^\dual$.

Let us set $\hat h=e^{2u}h^\#$ for some function $u$ on $S$.
Since the operator $\beta_\tau=\En+\tau b$ is a self-adjoint solution of 
the %% added
Codazzi equation,
the curvature of $h^\#$ is $K^\#=-\det(\beta_\tau)^{-1}$. 
Notice that $\det\beta_\tau=1+\tau^2+\tau\tr b\geq(1+\tau)^2$ so $K^\#\geq -(1+\tau)^{-2}$
and the equality holds only at points where $b=\En$.

The Liouville equation reads
\[
\Delta_{h^\#}u= e^{2u}+K^\#\geq e^{2u}-(1+\tau)^{-2}~.
\]
By the maximum principle we deduce that 
\begin{equation}\label{eq:curvest}
e^{2u}\leq (1+\tau)^{-2}~,
\end{equation}
 and if the equality holds at some points, then
$b=\En$ everywhere.

Now
we have $\hat h=e^{2u}h_c((\En+b)^{-2}(\En+\tau b)^2\bullet,\bullet )$,
so
\[
E(\Id;c,\hat{h})=\frac{1}{2}\int_S e^{2u}\tr[(\En+\tau b)^2(\En+b)^{-2}]\det(\En+b)\da_h~.
\]
On the other hand
\[
\det(\En+b)\tr[(\En+\tau b)^2(\En+b)^{-2}]=\det(\En+b)^{-1}\tr[(\En+\tau b)^2(\En+b^{-1})^2]
\]
where the last equality holds since $\det b=1$.
But %% Now 
we have
\begin{align*}
(\En+\tau b)^2(\En+b^{-1})^2 &=[(1+\tau)\En+\tau b+b^{-1}]^2=\\
&=(1+\tau)^2 \En+(\tau b)^2+b^{-2}+2(1+\tau)\tau b+2(1+\tau)b^{-1}+2\tau \En~.
\end{align*}
Taking the trace we deduce
that %% added
\[
\tr[(\En+\tau b)^2(\En+b^{-1})^2]=(1+\tau)^2[2+2\tr(b)+\tr(b^2)]-\tau[\tr(b^2)-2]~.
\]
Using (\ref{eq:curvest}) in this identity, we
obtain %% we deduce
%\begin{equation}
%\begin{array}{l}
\begin{align}\label{eq:enestgr}
E([c],[\hat h]) & \leq E(\Id, c,\hat{h})=\frac{1}{2} \int_Se^{2u}\tr[(\En+\tau b)^2(\En+b)^{-2}]\det(\En+b)\da_h\leq \\
& \leq \frac{1}{2}\int_S
\left\{[2+2\tr(b)+\tr(b^2)]-2\tau(1+\tau)^{-2}[\tr(b^2)-2]
\right\} \det(\En+b)^{-1}\da_h~. 
\notag
\end{align}
%\end{array}
%\end{equation}

On the other hand,
\begin{eqnarray*} %% I've transformed this into an eqnarray, it takes more space but is more readable.
E([c],[h]) & = & E(\Id;c,h) =\\
& = & \frac{1}{2}\int_S \tr[(\En+b)^{-2}]\det(\En+b)\da_h =\\
%& = & \frac{1}{2}\int_S\det(\En+b)\tr[(\En+b)^{-2}] =\\
& = & \frac{1}{2}\int_S\tr[(\En+b^{-1})^2]\det(\En+b)^{-1}\da_h =\\
%& = &\frac{1}{2}\int_S\tr[(\En+b^{-1})^2]\da_h =\\
& = & \frac{1}{2}\int_S\tr[\En+2b^{-1}+b^{-2}]\det(\En+b)^{-1}\da_h =\\
& = & \frac{1}{2}\int_S[2+2\tr(b)+\tr(b^2)]\det(\En+b)^{-1}\da_h~.
\end{eqnarray*}
%% So c
Comparing this identity with (\ref{eq:enestgr}), we get that
\[
E([c],[\hat h])\leq E([c],[h])-\tau(1+\tau)^{-2}\int[\tr(b^2)-2]\det(\En+b)^{-1}\da_h~,
\]
and this completes the proof.
\end{proof}

In order to conclude that the topological degree of $\cG=sgr'_{s}(h,\bullet)$ is $1$,
it is sufficient to prove that $d \cG$ at $[h]$ is non-degenerate. 

Consider a one-parameter family of Labourie operators $t\mapsto b_t$ %% added parenthesis
such that $b_0=\En$
and $\dot b$ is non-zero. Notice that $\dot b$ is a traceless self-adjoint solution of
the %% added
Codazzi equation. 
Now consider the path of complex structure $\wti{J}_t$ compatible with
$\wti{h}^\#_t=h(\beta_t,\beta_t)$ and $\beta_t=\cosh(s/2)\En+\sinh(s/2) b_t$.

Since $\wti{J}_t=\beta_t^{-1}J\beta_t$, 
%% it turns out that 
the derivative of $\wti{J}_t$ at $t=0$ is the traceless $h$-self-adjoint operator
$\dot{\wti{J}}=[\wti{J},\dot\beta]=2J\dot b$. 

The derivative of the path $[\wti{J}_t]\in\Teich$ at $t=0$  is the projection of
$\dot{\wti{J}}$ to $T_{[J]}\Teich$, through the natural map $\mathcal A\rightarrow\mathcal T$.
By  \cite{fischer-tromba:cell}, %% I think we should give a citation here.  
Codazzi solutions in $T_J\mathcal A$ form a complement
of the kernel of the projection $T_J\mathcal A\rightarrow T_{[J]}\Teich$. Since
$\dot{\wti{J}}$ lies in this subspace, then it projects to a non-zero vector.

\end{subsection}

%\begin{subsection}{The map $sgr_{s}{h,\bullet}$ is a diffeomorphism.}
%
%Now fix $h^\dual\in\cT$ and $s>0$ and consider the map
%\[
%sgr_{s,h^\dual}:\cT\lra\cT
%\]
%We claim that $sgr_{s,h^\dual}$ is a diffeomorphism.
%
%Notice that $sgr_{s,h^\dual}(h_t)$ is represented
%by the metric $\wti{h}_t=h^\dual(\wti{\beta}_t,\wti{\beta}_t)$, 
%where $\wti{\beta}_t=\sinh(s/2)\En+\cosh(s/2)b_t^{-1}$.
%
%The proof that $sgr_{s,h^\dual}$ is a local diffeomorphism
%is analogous to the case dealt with above, just reversing the roles of
%$h$ and $h^\dual$.

%To show that $sgr_{s,h^\dual}$ is proper, we can consider
%a divergent sequence $[h_n]$ and set $\wti{\beta}'_n=\tau \En+b_n^{-1}$
%with $\tau=\tanh(s/2)$ and $\wti{h}'_n=h^\dual(\wti{\beta}'_n,\wti{\beta}'_n)$.
%Denoting by $h_{c_n}$ the metric $h^\dual(\En+b_n^{-1},\En+b_n^{-1})$
%representing the center of $(h_n,h^\dual)$ and by
%$f_n^\dual:(S,h_{c_n})\rar (S,h^\dual)$ the harmonic map isotopic to the
%identity, we easily obtain that $\da_{h_{c_n}}\leq \tau^{-1}\da_{\wti{h}'_n}$
%and that $\wti{h}'_n\leq h_{c_n}$. Hence,
%\[
%E(\Id;h_{c_n},h^\dual)\leq E(f^\dual_n;h_{c_n},h^\dual)\leq \tau^{-1} E(f^\dual_n;\wti{h}'_n,h^\dual)
%\]
%and we conclude as before.

%\end{subsection}

\subsection{Hyperbolic ends}
\label{ssc:ends}

In this section we use the parameterization of landslide and smooth grafting
by the upper half-plane, so that we use the notations $SGr'$ and $sgr'$ as
in \cite{cyclic}.
Recall from \cite[Section 5]{cyclic} that two smooth grafting maps can be 
considered. One, $SGr'$, takes its values in $\cCP$, the space of complex
projective structures on $S$, while the other, $sgr'$, goes to the Teichm\"uller
space of $S$.

Given two hyperbolic metrics $h,h^\dual\in \cT$ and $s>0$, there is a unique 
equivariant embedding $\sigma$ of the universal cover
$\tilde{S}$ of $S$ inside $\Hyp^3$ with induced metric $\cosh^2(s/2)\tilde{h}$
and third fundamental form $\sinh^2(s/2)\tilde{h}^\dual$.
Then $SGr'_s(h,h^\dual)$ is the 
complex projective structure induced on $S$ from the complex projective
structure on $\partial_\infty \Hyp^3$ by the hyperbolic Gauss map. The 
complex structure $sgr'_s(h,h^\dual)$ is the complex structure underlying $SGr'_s(h,h^\dual)$.

The equivariant embedding $\sigma$ is locally convex by the Gauss formula
(the Gaussian curvature of the induced metric is $-1/\cosh^2(s/2)>-1$) so that
the quotient of the image of $\sigma$ by the image of its associated representation
of $\pi_1S$ is a convex surface %$\Sigma$
in a hyperbolic end $\hend$. This hyperbolic 
end is uniquely determined by $h, h^\dual$ and $s$, and its conformal structure
at infinity is equal to $sgr'_s(h,h^\dual)$. 

\begin{proof}[Proof of Theorem \ref{tm:hyperbolic}]
According to Theorem \ref{tm:sgr},
%--- restated in terms of a parameterization by the upper half-plane rather than the disk ---
the map 
$sgr'_s(h,\bullet): \cT\rightarrow \cT$ is 
surjective. %% a homeomorphism. 
This means precisely
that, given $h$ and $c\in \cT$, there is a %% unique 
$h^\dual\in \cT$ such that 
$sgr'_s(h,h^\dual)=c$, so that there is a %% unique 
hyperbolic end with complex
structure at infinity $c$ containing a surface of constant curvature 
$-1/\cosh^2(s/2)$ with induced metric homothetic to $h$.
\end{proof}

%% I've added what follows

We now recall briefly some key points concerning de Sitter domains of dependence, so
as to be able to prove Theorem \ref{tm:deSitter}. A de Sitter domain of dependence is
a (non-complete) 3-dimensional manifold locally modelled on the de Sitter space, which
is future-complete and globally hyperbolic. 

De Sitter domains of dependence are in one-to-one correspondence with hyperbolic
ends. One way to see this correspondence is that, given a hyperbolic end, there is
a unique de Sitter domain of dependence with the same fundamental group and the
same representation of the fundamental group into $\PSL(2, \C)$. 

However it is perhaps simpler here to characterize this correspondence in terms
of convex embedded surfaces. Let $\hend$ be a hyperbolic end, and let $S$ be a
locally strictly convex surface in $\hend$ which bounds a convex domain. The
universal cover $\St$ of $S$ is then a complete, locally convex surface in $\Hyp^3$
invariant under the action of the fundamental group of $\hend$. The dual surface
$\St^\dual$ is then a strictly future-convex, space-like surface in the de Sitter space $\dS^3$, 
also invariant under the action of the fundamental group of $\hend$ but now considered 
as acting on the de Sitter space. The action of $\pi_1 S$ is free and properly 
discontinuous on a convex domain $\tDom$ in $\dS^3$ containing $\St^\dual$, and the quotient
is the de Sitter domain of dependence $\Dom$ corresponding to $\hend$.

The conformal structure at infinity of $\Dom$ is the same as the conformal structure
at infinity of $\hend$. It can be defined in terms of the conformal structure at 
future infinity of $\Dom$, or in terms of the quotient by $\pi_1 S$ of the boundary
at infinity of $\tDom\subset \dS^3$.

\begin{proof}[Proof of Theorem \ref{tm:deSitter}]
Let $h^\dual,h'\in \cT$, and let $K^\dual\in (-\infty,0)$. Let $K:=K^\dual/(1-K^\dual)$, so that $K\in (-1,0)$ 
--- thus, $K$ is the curvature of a surface in $\Hyp^3$ dual to a surface of curvature $K^\dual$
in $\dS^3$.
The second part of Theorem \ref{tm:sgr} implies that there exists $h\in \cT$ such that 
$sgr'_s(h,h^\dual)=h'$, where $s$ is chosen so that $-1/\cosh^2(s/2)=K$. 

This means precisely
that there exists a hyperbolic end $\hend$ containing a surface $S$ with constant curvature $K$, 
with induced metric homothetic to $h$, third fundamental form homothetic to $h^\dual$, and
conformal structure at infinity equal to $h'$.

But then the de Sitter domain of dependence corresponding to $\hend$ contains a surface $S^\dual$ --- dual
to $S$ --- with constant curvature $K^\dual$, induced metric proportional to $h^\dual$ and third fudamental
form proportional to $h$. This proves the theorem.
\end{proof}
 %smooth grafting (section 6)

\section{The smooth grafting map is symplectic}

In this section we consider symplectic properties of the smooth
grafting map, and prove Proposition \ref{pr:parameterization} and
Theorem \ref{tm:sgr-symplectic}.  A key point in the proof of
Proposition \ref{pr:parameterization} will actually be a consequence
of the symplectic arguments occuring in the proof of Theorem
\ref{tm:sgr-symplectic}.

\subsection{The renormalized volume beyond a $K$-surface}

A Poincar\'e-Einstein manifold is a manifold $\hend$ diffeomorphic to the
interior of a compact manifold with boundary $\overline{\hend}$, with a Riemannian
metric $g$ which is Einstein and can be written near the boundary as
$$ g = \frac{\gb}{\rho^2}~, $$ where $\gb$ is a smooth metric on $\overline{\hend}$
and $\rho$ is a smooth function on $\overline{\hend}$ vanishing on the boundary and
with $\|d\rho\|_{\gb}=1$ on $\partial \hend$. %% added subscript
In dimension 3,
Poincar\'e-Einstein manifolds are the same as convex co-compact
hyperbolic manifolds.

The volume of a Poincar\'e-Einstein manifold is always
infinite. However it is possible to define a ``renormalized volume''
which is finite and has interesting properties, see
\cite{graham-witten}. In even total dimension, this renormalized
volume is well-defined, while in odd total dimension it depends on the
choice of a metric in the conformal class at infinity.

For quasifuchsian manifolds, in total dimension 3, it makes sense to
choose as the metric at infinity the (unique) hyperbolic metric in the
conformal class at infinity. The renormalized volume which is then
obtained is intimately related to the Liouville functional introduced
by Takhtajan and Zograf \cite{TZ-schottky,takhtajan-zograf:spheres}
for the Schottky uniformization and for the punctured sphere, later
extended to higher genus surfaces \cite{takhtajan-teo}.

Here we follow the analysis of the renormalized volume of hyperbolic
3-manifolds developed in \cite{volume,review}.  The argument we use %% used
is strongly related to that used in \cite{cp}, so we only sketch the
main points.  We consider a hyperbolic end $\hend$ containing a convex
surface $S$ of constant curvature, isotopic to the boundary at
infinity.

Consider a foliation of a neighborhood of infinity in $\hend$ by
equidistant surfaces $(\Sigma_t)_{t\geq t_0}$, with all leaves between
$S$ and the boundary at infinity of $\hend$. Let $I_t, \II_t, \III_t$ and
$\da_t$, respectively, be the induced metric, second fundamental form,
third fundamental form, and area form of $\Sigma_t$, and by $I,\II,
\III$ and $\da$ the corresponding quantities on $S$. For both $S$ and
$\Sigma_t$ 
we %% added
use the unit normal pointing towards infinity in $\hend$ when
defining $\II$. We also call $H$ (resp. $H_t$) the mean curvature of
$S$ (resp. $\Sigma_t$), that is, $H=\tr_I\II$.

\begin{defi}
For all $t\geq t_0$ we denote by $V_t$ the volume of the domain of $\hend$ bounded
by $S$ and $\Sigma_t$, and set
$$ W_t = V_t -\frac{1}{4} \int_{\Sigma_t} H_t \da_t + \frac{1}{2}\int_S H\da~. $$
\end{defi}

The following proposition is a direct consequence of the main result
of \cite{sem,sem-era}.

\begin{prop} \label{pr:schlafli}
In a first-order deformation of $\hend$, the first-order variation of
$W_t$ is given by:
\begin{equation}
  \label{eq:schlafli}
  \frac{dW_t}{dt} = \frac 14 \int_{\Sigma_t}
\left(
\frac{dH_t}{dt} + \langle \frac{dI_t}{dt}, \II_t - \frac{H_t}{2} I_t\rangle 
\right)
\da_t
- \frac 12\int_S \langle \frac{dI}{dt},\II - H I\rangle \da~. 
\end{equation}
\end{prop}

\begin{proof} %[Sketch of the proof]
According to \cite[Theorem 1]{sem}, in any first-order deformation of $\hend$,
$$ \frac{dV_t}{dt} = \frac 12\int_{\Sigma_t}
\left(
\frac{dH_t}{dt} + \frac 12\langle \frac{dI_t}{dt}, \II_t\rangle
\right)
\da_t
- \frac 12\int_S
\left(
\frac{dH}{dt}+\frac 12\langle \frac{dI}{dt}, \II\rangle
\right)
\da~. $$
However an elementary computation shows that
$$ \frac{d}{dt}\int_S H\da = \int_S
\left(
\frac{dH}{dt}+\frac{H}{2}\langle \frac{dI}{dt},I\rangle
\right)
\da~, $$
and similarly for $\Sigma_t$. The result follows by a simple computation.
\end{proof}

\begin{cor}
The derivative of $W_t$ with respect to $t$ is given by
$$ \frac{dW_t}{dt} = -\pi \chi(S)~. $$  
\end{cor}

\begin{proof}
Since the surfaces $\Sigma_t$ are equidistant, we have
$$ \frac{dI_t}{dt} = 2\II_t~,~~ \frac{dB_t}{dt} = \En-B_t^2~, ~~ \frac{dH_t}{dt} = 2-\tr(B_t^2)~. $$
Replacing this in Equation (\ref{eq:schlafli}) leads to
\begin{eqnarray*}
\frac{dW_t}{dt} & = & \frac 14
\left(
\int_{\Sigma_t} 2-\tr(B_t^2) + 2\tr(B_t^2) - H_t^2
\right)
\da_t\\
& = & \frac 12\int_{\Sigma_t} \left( 1-\det(B_t) \right) \da_t \\
& = & \frac 12\int_{\Sigma_t} (-K_t) \da_t~, 
\end{eqnarray*}
where $K_t$ is the curvature of $I_t$. The result follows by the Gauss-Bonnet formula.
\end{proof}

\begin{defi}
We define the renormalized volume above $S$ by 
$$ W := W_t + \pi \chi(S)t~, $$
which is clearly independent of the choice of $t\geq t_0$.
\end{defi}

Note that $W$ can be defined simply as $W_0$ if $t_0\leq 0$, %% replaced t by t_0
however this is not always the case. 

This quantity $W$ depends only on the hyperbolic end $\hend$, on $S$, and on the equidistant
foliation of $\hend$ near infinity. Below we defined another quantity $\cW$, depending only on
$\hend$ and on $S$, obtained by taking a special, canonically defined foliation near infinity.

\subsection{The data at infinity of a hyperbolic end} %% added "a"

Recall that if $\Sigma$ is a surface in hyperbolic 3-space, and if $\Sigma_t$ is
a surface at constant distance $t$ from $\Sigma$, the induced metric on $\Sigma_t$ can
be expressed in terms of the induced metric $I$ and the shape operator $B$ of $\Sigma$ as:
$$ I_t(x, y) = I
\big(
(\cosh(t)\En + \sinh(t)B)x, (\cosh(t)\En + \sinh(t)B)y
\big)~. $$

It follows directly that the induced metrics $I_t$ have a simple asymptotic development 
as $t\rightarrow \infty$, which can be written as:
$$ I_t = e^{2t}I_\infty + 2\II_\infty+e^{-2t}\III_\infty~, $$
where $I_\infty, \II_\infty$ and $\III_\infty$ are bilinear symmetric forms on $S$ which can be expressed
quite simply in terms of $I_t$ and $B_t$ for any given value of $t$. We call $\nabla^\infty$ the Levi-Civita connection of $I_\infty$,
$K_\infty$ its curvature, $B_\infty:TS\rightarrow TS$ the linear map
which is self-adjoint for $I_\infty$ 
and such that 
$$\qquad\II_\infty(x,y)=I_\infty(B_\infty x,y)
\quad \forall p\in S,\ \forall x,y\in T_p S
$$
and $H_\infty=\tr(B_\infty)$.

The following lemma recalls some properties of this asymptotic expansion, details can
be found in \cite{volume,review}.

\begin{lemma} \label{lm:items}
  \begin{enumerate}
  \item $I_\infty$ is in the conformal class at infinity of $\hend$,
  \item $I_\infty$ and $B_\infty$ satisfy the Codazzi equation, $d^{\nabla^\infty}B_\infty=0$, and a modified
version of the Gauss equation, $K_\infty=-H_\infty$,
  \item $I_\infty$ and $\II_\infty$ together determine uniquely $\hend$,
  \item any metric $I_\infty$ in the conformal class at infinity of $\hend$ is obtained from
a unique foliation of a neighborhood of infinity in $\hend$ by equidistant surfaces.
  \end{enumerate}
\end{lemma}

A key point is that there are simple formulas relating the data $I_t, B_t, \II_t, \III_t$
on a surface $\Sigma_t$ to the corresponding data at infinity, see \cite[Section 5]{volume} or \cite{review}.
This leads in particular to the following analog of Proposition \ref{pr:schlafli}, see \cite[Lemma 6.1]{volume}.

\begin{prop} \label{pr:schlafli*}
In a first-order deformation of $\hend$ and of the foliation $(\Sigma_t)_{t\geq t_0}$, 
the first-order variation of $W$ is given by:
\begin{equation}
  \label{eq:schlafli*}
  \frac{dW}{dt} = - \frac 14 \int_{\partial_\infty S} 
\left(
\frac{dH_\infty}{dt} + \langle \frac{dI_\infty}{dt}, \II_\infty - \frac{H_\infty}2 I_\infty\rangle
\right)
\da_\infty
- \frac 12\int_S \langle \frac{dI}{dt},\II-H I\rangle \da~. 
\end{equation}
\end{prop}

\subsection{Smooth grafting is symplectic}

Point (4) of Lemma \ref{lm:items} in particular is used in the next definition.

\begin{defi}
We let $\cW$ be %% as 
the value of $W$ when the foliation $(\Sigma_t)_{t\geq t_0}$ is the unique
foliation such that $I_\infty$ is the hyperbolic metric at infinity.
\end{defi}

With this definition, Proposition \ref{pr:schlafli*} has a direct consequence. Let $(\II_\infty)_0$ be
the traceless part of $\II_\infty$ (with respect to $I_\infty$). 

\begin{prop} \label{pr:cW}
In a first-order deformation of $\hend$, the first-order variation of $\cW$ is given by
\begin{equation}
  \label{eq:cW}
\frac{d\cW}{dt} = -\frac 14\int_{\partial_\infty S} \langle \frac{dI_\infty}{dt}, (\II_\infty)_0\rangle_{I_\infty} \da_\infty
- \frac 12\int_S \langle \frac{dI}{dt},\II-HI\rangle_I \da_I~.
\end{equation}
\end{prop}

\begin{proof}
This follows directly from Proposition \ref{pr:schlafli*} using the fact that $H_\infty=-K_\infty=1$,
that $\tr_{I_\infty}\II_\infty=H_\infty$, and that $\tr_I\II=H$. 
\end{proof}

Both terms occuring in (\ref{eq:cW}) can be interpreted in an interesting way. 
%First we recalla basic remark which will be useful below. 

%\begin{remark} \label{rk:divergence}
%Let $c$ be a 1-form on $S$ with values in $TS$ which, considered as an endomorphism of $TS$,
%is self-adjoint. Then $d^\nabla c=0$ if and only if 
%$\delta^\nabla (c-\tr(c)\En)=0$, where $\En$ is the identity on $TS$.
%\end{remark}
%
%\begin{proof}
%Let $(e_1, e_2)$ be an orthonormal frame on $S$ and let $v$ be a vector field on $S$, then
%\begin{eqnarray*}
%(\delta^\nabla (c-\tr(c)\En))(v) & = & \sum_i \langle (\nabla_{e_i}c)(e_i)-d\tr(c)(e_i)e_i,v\rangle \\
%& = & \sum_i \langle (\nabla_{e_i}c)(v),e_i\rangle -d\tr(c)(v) \\
%& = & \sum_i \langle (\nabla_{v}c)(e_i)+(d^\nabla c)(e_i,v),e_i\rangle -d\tr(c)(v) \\
%& = & d\tr(c)(v) + \sum_i \langle (d^\nabla c)(e_i,v),e_i\rangle - d\tr(c)(v) \\
%& = & \sum_i \langle (d^\nabla c)(e_i,v),e_i\rangle~,
%\end{eqnarray*}
%and the result follows.
%\end{proof}

We need to identify the image by $d_1F$ of a point $(h,h^\dual)\in \cT\times \cT$.

\begin{lemma} \label{lm:beta}
Let $(h,h^\dual)\in \cT\times \cT$, and let $b$ be the Labourie operator of $(h,h^\dual)$. Then
$d_1F(h,h^\dual)=(h,\beta)\in T^*\cT$, where $\beta\in T^*_h\cT$ is defined, for any first-order
variation $\hd$ of $h$, by
\begin{equation} \label{eq:beta}
\beta(\hd)=-\frac{1}{2}\int_S \langle \hd, h(b\bullet,\bullet)-\tr(b) h\rangle_h \da_h~.
\end{equation}
% I removed a - sign: correct?
\end{lemma}

\begin{proof}
If we put $h_t=h(\alpha_t, \alpha_t)$ with $\alpha_t$ positive self-adjoint, then
$\hd=2 h(\dot\alpha\bullet, \bullet)$.

In particular we have 
\[
\frac{1}{2} \langle \hd, h(b\bullet,\bullet)-\tr(b) h\rangle_h =\tr(\dot\alpha b)-\tr(\dot\alpha)\tr(b)~.
\]
The result then follows from (\ref{eq:1der}) and Remark \ref{rk:gen}.
%Let $h_t$ be a smooth family of hyperbolic metrics, with $h=h_0$.
%The unique harmonic map $f_t:(S,h)\rar (S,h_t)$ pulls
%the metric $h_t$ back to $f_t^*(h_t)=h(\alpha_t\bullet,\bullet)+e(t)h$,
%where $\alpha_t$ is $h$-self-adjoint, traceless and satisfies $d^{\nabla}(\alpha_t)=0$.
%As $e(t)=1+o(t)$, taking the derivative at $t=0$ we obtain
%$$ \hd = h(\alphad\bullet,\bullet) - 2\delta^* v~, $$
%where $\alphad$ is traceless, $h$-self-adjoint and satisfies $d^\nabla \alphad=0$, and $v=df_t/dt|_{t=0}$ is a vector field
%on $S$. 
%
%The second part of Proposition \ref{prop:1der} shows that (\ref{eq:beta}) holds for the
%component $h(\alphad\bullet,\bullet)$. The term $-2\delta^* v$ represents a trivial deformation 
%when considered in $\cT$, so the left-hand side of (\ref{eq:beta}) vanishes, while the
%right-hand side vanishes because $h(b\bullet,\bullet)-\tr(b) h$ is divergence-free by 
%Remark \ref{rk:divergence}. 
%
%%{\bf NB explain why there is no trace part -- because the metrics %are hyperbolic.}
\end{proof}

To give a geometric interpretation of this fact, we note that the smooth grafting map
$SGr'_s:\cT_h\times \cT_{h^\dual}\rightarrow \cCP$ can be decomposed as follows. 
Let $\Ends$ be the
space of hyperbolic ends. There is a natural homeomorphism $\partial_\infty:\Ends\to \cCP$
sending a hyperbolic end to its complex projective structure at infinity. Moreover, 
each hyperbolic end $\hend\in \Ends$ contains a unique convex surface $S_{\hend}$ with constant
curvature $-1/\cosh^2(s/2)$, and we can consider the map $\kappa_s:\Ends\to \cT_h\times \cT_{h^\dual}$
sending $\hend$ to $(h,h^\dual)$, where $h$ and $h^\dual$ are the hyperbolic metrics homothetic
respectively to the induced metric and to the third fundamental form of $S_{\hend}$.
By construction,
the following diagram commutes
%$SGr'_s=\partial_\infty\circ \kappa_s^{-1}$.
\[
\xymatrix{
\Ends \ar[drr]_{\partial_\infty}
\ar[rr]^{\kappa_s\quad} && \cT_h \times \cT_{h^\dual} \ar[d]^{SGr'_s}
\ar[rr]^{\quad d_1 F} && T^*\cT_h \\
&& \ar[rr]^{\Sch}\cCP && T^*\cT_\infty
}
\]
where $\Sch$ is the Schwarzian derivative with respect to the Fuchsian section.

\begin{defi}
We denote by $\lambda$ the Liouville form on $T^*\cT$.  
\end{defi}

We can now identify the second integral in (\ref{eq:cW}).

\begin{cor} \label{cr:liouville1}
The pull-back 
of the Liouville form on $T^* \cT_h$
through $d_1F\circ \kappa_s:\Ends\rightarrow T^* \cT_h$ 
is given by
$$ (d_1F\circ \kappa_s)^*\lambda = \frac 1{\sinh(s)}\int_S \langle \d I, \II-HI\rangle_I \da_I~. $$
\end{cor}

\begin{proof}
This follows directly from Lemma \ref{lm:beta}, taking into account the homothetic
factors: $I=\cosh^2(s/2)h$, $\II=\cosh(s/2)\sinh(s/2)h(b\bullet, \bullet)$ and $\III = \sinh^2(s/2)h(b\bullet,b\bullet)$.
\end{proof}

Finally we can identify the first integral in  (\ref{eq:cW}). The following lemma is another
way to state Lemma 8.3 in \cite{volume}.

\begin{lemma} \label{lm:eta*}
The pull-back of the Liouville form of $T^*\cT_\infty$ through the map $\Sch\circ \partial_\infty:\Ends \to 
T^*\cT_\infty$ is the 1-form given by
$$ (\Sch\circ \partial_\infty)^*\lambda = - \int_{\partial_\infty S} \langle {\d I_\infty}, (\II_\infty)_0\rangle_{I_\infty} \da_\infty~. $$
\end{lemma}

\begin{proof}[Proof of Theorem \ref{tm:sgr-symplectic}]
Consider the two 1-forms defined on $\Ends$ by
$$ \eta(\hend) = \int_S \langle \d I,\II-HI\rangle_I\da_I~,
\qquad
\eta_\infty(\hend) = \int_{\partial_\infty \hend} \langle {\d I_\infty},(\II_\infty)_0\rangle_{I_\infty}\da_{I_\infty}~. $$
%% I united the two formulas on a line
According to Proposition \ref{pr:cW}, we have on $\Ends$
$$ d\cW = -\frac 14 \eta_\infty - \frac 12 \eta~, $$
so that 
$$ d\eta_\infty + 2d\eta=0~. $$

However Corollary \ref{cr:liouville1} shows that
$$ 2\eta = 2\sinh(s) (d_1F\circ \kappa_s)^* \lambda~, $$
while Lemma \ref{lm:eta*} indicates that 
$$ \eta_\infty = -  (\Sch\circ \partial_\infty)^*\lambda~. $$
So
$$ 2\sinh(s) (d_1F\circ \kappa_s)^* \lambda = (\Sch\circ \partial_\infty)^*\lambda~, $$
and, calling $\omega_{can}=d\lambda$ the cotangent symplectic form on $T^*\cT$, we have
$$ 2\sinh(s) (d_1F\circ \kappa_s)^* \omega_{can} = (\Sch\circ \partial_\infty)^*\omega_{can}~. $$
This proves the result.
\end{proof}

\subsection{Proof of Proposition \ref{pr:parameterization}}

The proof of Proposition \ref{pr:parameterization} is based on the following two lemmas.

\begin{lemma} \label{lm:dFinjective}
The differential of the map $d_1F:\cT\times \cT\to T^*\cT$ is an isomorphism at each point.
\end{lemma}

\begin{lemma} \label{lm:dFproper}
The map $d_1F:\cT\times \cT\to T^*\cT$ is proper.
\end{lemma}

It follows from Lemma \ref{lm:dFinjective} that $d_1F$ is a local homeomorphism. Since it is
proper, it is a covering. But $T^*\cT$ is simply connected and $\cT\times \cT$ is connected,
so $d_1F$ is a homeomorphism. This concludes the proof of the proposition.

We now turn to the proofs of those lemmas.

\begin{proof}[Proof of Lemma \ref{lm:dFinjective}]
We have seen above that 
$$ d\eta_\infty = -  (\Sch\circ \partial_\infty)^*\omega_{can}~. $$
Since both $\Sch$ and $\partial_\infty$ are diffeomorphisms, it follows that 
$d\eta_\infty$ is non-degenerate. 

However we have also seen that $d\eta_\infty + 2d\eta=0$, so that $d\eta$ is
also non-degenerate. Since 
$$ 2d\eta = 2\sinh(s) (d_1F\circ \kappa_s)^* \omega_{can}~, $$
and $\kappa_s$ is onto,
both $d_1F$ and $\kappa_s$ have differentials of maximal rank at each point. 
\end{proof}

\begin{proof}[Proof of Lemma \ref{lm:dFproper}]
Let $(h_n,h^\dual_n)_{n\in\N}$ be a diverging sequence in $\cT\times\cT$. If $h_n$ diverges, so does $(d_1 F)_{(h_n,h^\dual_n)}$.
Hence, up to extracting a subsequence, we can assume that $h_n\rar
h\in\cT$ and that 
$\th_n \ell_{h^\dual_n}\to\iota(\lambda,\bullet)$, 
where $\th_n\to 0$ and $\lambda$ is a nonzero measured lamination.

Now, the functions $F_n=F_{h^\dual_n}:\cT\to \R$ determine a
sequence $(\th_n F_n)_{n\in\N}$ that converges to $\ell_\lambda$, uniformly on the compact
subsets of $\cT$.
As the functions $F_n$ and $\ell_\lambda$ are real-analytic, the convergence is also 
$C^\infty$ on the compact subsets of $\cT$.
Hence, $\th_n d_{h_n} F_n\to d_h\ell_\lambda$ and so
$(d_1 F)_{(h_n,h_n^\dual)}=d_{h_n} F_n$ diverges.
\end{proof}
 %SGr is symplectic etc. (section 7)

\section{Extension to the boundary}

Fix $c\in\cT$ and let $\cQ_c$ be the space of holomorphic quadratic
differentials on $(S,c)$. Then consider the Sampson-Wolf map
\[
\SW_c:
\cQ_c \lra \cT
\]
that assigns to $q$ the class of the unique hyperbolic metric $h=\SW_c(\qd)$ such that the identity $\Id:(S,c)\rar (S,h)$ is harmonic with Hopf differential equal to $\frac{1}{4}\qd$.

\begin{theorem}[Sampson \cite{sampson:78}, Wolf \cite{wolf:teichmuller}]
The map $\SW_c$ is a real-analytic diffeomorphism.
\end{theorem}

It is well-known since Thurston \cite{FLP} that it is possible to
produce a compactification $\ol{\cT}$ of $\cT$ by adding the space of projectively measured laminations $\PP\cML$ at infinity.

Here we recall that to every nonzero holomorphic quadratic differential $\qd$
on a Riemann surface $S$ we can attach
a horizontal foliation $\cF_+(\qd)$ (resp. vertical foliation $\cF_-(\qd)$)
along which $\qd$ restricts as a positive-definite (resp. negative-definite) real quadratic form, which is
singular at the points where $\qd$ vanishes.
Moreover, $\cF_+(\qd)$ (resp. $\cF_-(\qd)$) comes endowed with a measure $|\mathrm{Im}\sqrt{\qd}|$ (resp. $|\mathrm{Re}\sqrt{\qd}|$)
transverse to its leaves. (A more extensive discussion can
be found in \cite{strebel:quadratic}.)

To every measured foliation $\cF$ one can associate a
measured lamination, intuitively by ``straightening'' the leaves
of $\cF$ to geodesics with respect to some hyperbolic metric.
We notice that the measured laminations $\l_\pm(\qd)$ associated
to $\cF_\pm(\qd)$ fill the surface in the following sense.

\begin{defi}
A couple $(\l_+,\l_-)$ of measured laminations on $S$ is
filling if $i(\l_+,\mu)+i(\l_-,\mu)>0$ for every lamination $\mu\neq 0$. We denote by $\cFML\subset\cML\times\cML$ the open locus of filling laminations.
\end{defi}

The exact correspondence between holomorphic quadratic differentials and filling measured laminations relies on the following result.

\begin{theorem}[Hubbard-Masur \cite{hubbard-masur:acta79}]
\label{thm:hm}
The map $\cQ\rar\cFML\cup\{0\}$ defined as $\qd\mapsto (\l_+(\qd),\l_-(\qd))$
is a homeomorphism.
\end{theorem}

In order to extend our construction to some boundary at infinity,
the following result will play a key role.

\begin{theorem}[Wolf \cite{wolf:teichmuller} \cite{wolf:R-trees}]
Let $\ol{\cQ}_c$ be the compactification of $\cQ_c$ obtained by adding the sphere at infinity $\partial \cQ_c=(\cQ_c\setminus\{0\})/\R_+$.
Then $\SW_c$ extends as a homeomorphism
\[
\ol{\SW}_c:
\ol{\cQ}_c \lra \ol{\cT}
\]
by defining $\ol{\SW}_c([q])=[\cF_-(q)]$
for every $[q]\in\partial\cQ_c$.
\end{theorem}

It will be more practical to work with a de-homogeneized
version of the above result.
Indeed, in Thurston's picture the space $\cY=\cT\times\R_{<0}$
of metrics of constant negative curvature on $S$ (up to isotopy)
can be completed as $\ol{\cY}=(\cML\times\{-\infty\})\cup\cY$ by adding a copy of $\cML$.

\begin{cor}
The following map
\[
\qquad\what{\SW}_c:
\xymatrix@R=0in{
\cQ_c\times[-\infty,0) \ar[rr] && \ol{\cY} \\
(\qd,K) \ar@{|->}[rr] && (\SW_c(|K|\qd),K) & \text{if $K\in(-\infty,0)$}\\
(\qd,-\infty)\ar@{|->}[rr] && (\cF_-(\qd),-\infty) & \text{if $K=-\infty$}
}
\]
is a homeomorphism.
\end{cor}

In order to study the behavior of the landslide flow as the metrics degenerate, we consider the space $\cDY=\cT\times\cT\times\R_{<0}$ of couple of metrics with the same constant negative curvature on $S$ (up to isotopy) and the partial completion $\ol{\cDY}=\cDY\cup
(\cFML\times\{-\infty\})$.

Now consider
\[
\qquad\what{\SW}:
\xymatrix@R=0in{
\cQ\times[-\infty,0) \ar[rr] && \ol{\cDY} \\
(c,\qd,K) \ar@{|->}[rr] && (\SW_c(|K|\qd),\SW_c(-|K|\qd),K) & \text{if $K\in(-\infty,0)$}\\
(c,\qd,-\infty)\ar@{|->}[rr] && (\cF_-(\qd),\cF_+(\qd),-\infty) & \text{if $K=-\infty$~.}
}
\]

\begin{prop}\label{prop:What}
The map $\what{\SW}$ is a homeomorphism.
\end{prop}

We recall that the extremal length of $c$ with respect to $\lambda$ depends real-analytically on $c\in\cT$ and it satisfies
%% as 
$\Extr_\lambda(c)=\|\qd\|=2E(c,\lambda)$, where $E(c,\lambda)$
is the energy of the harmonic map $f$ from $c$ to the $\R$-tree dual to $\lambda$, 
and $\frac{1}{4}\qd$ is the Hopf differential of $f$ and also
the unique holomorphic quadratic
differential on $c$ with $\cF_-(\qd)=\lambda$
(see for instance \cite{wolf:R-trees}).
%% I think we should add a reference of two for those facts.

\begin{lemma}\label{lemma:Ext}
Given $h',h\in\cT$ and let $E(h',h)$ be the energy of the unique hamornic map $(S,h')\rar(S,h)$ isotopic to the identity.
For every $h,h^\dual\in\cT$, the function $E(\bullet,h)+E(\bullet,h^\dual):\cT\rar\R_+$ is proper 
and achieves a unique minimum at the center $c$ of the couple $(h,h^\dual)$.
Similarly, if $(\lambda,\mu)\in\cFML$, then the function $\Extr_\lambda+
\Extr_\mu:\cT\rar\R_+$ is proper and achieves a unique minimum at
$c$, where $c$ is the conformal structure underlying the Hubbard-Masur quadratic
differential $\qd$ associated to $(\lambda,\mu)$.
\end{lemma}

\begin{proof}
Properness of the energy function is proven in Proposition~\ref{prop:energy}
and the remaining part of the first claim can be found in Theorem 1.10(iv) of \cite{cyclic}.

For the second statement, we have $\Extr_\lambda\geq \ell_\lambda^2
/(2\pi|\chi(S)|)$ by the definition of extremal length. As $\ell_\lambda+\ell_\mu$ is proper
\cite{kerckhoff:minima}, the same holds for $\Extr_\lambda+\Extr_\mu$.
Moreover, Gardiner's formula \cite{gardiner:acta84} gives
\[
d\Extr_\lambda|_{\bullet=c}=-\frac{\qd}{2}
\]
where $\frac{1}{4}\qd$ is the Hopf differential of the harmonic map from $c$ to the $\R$-tree dual to $\lambda$; in other words, $\qd$ is also the unique
holomorphic quadratic differential on $c$ whose vertical foliation
corresponds to $\lambda$.
Thus, if $d\Extr_\mu|_{\bullet=c}=-\frac{1}{2}\psi$, then
$c$ is a minimum if and only if $\qd=-\psi$ and so
$\mu$ corresponds to the horizontal foliation of $\qd$.
We conclude by Theorem~\ref{thm:hm}.
\end{proof}

\begin{remark}
The previous argument also shows that
the functions $E(\bullet,h)+E(\bullet,h^\dual)$ and
$\Extr_\lambda+\Extr_\mu$ have a unique
{\it local} minimum.
\end{remark}

\begin{proof}[Proof of Proposition~\ref{prop:What}]
By the above corollary, $\what{\SW}$ is continuous. Moreover,
$\what{\SW}$ is bijective and its inverse can be described as follows.

Given $(h,h^\dual,K)$ with $K\in(-\infty,0)$,
we can assume that $(h,h^\dual)$ are normalized representatives.
Then we let $c$ be the conformal structure underlying the 
metric $h+h^\dual$, so that $\Id:(S,c)\rar (S,h)$ and $\Id:(S,c)\rar(S,h^\dual)$ have Hopf differentials $\frac{1}{4}\qd$ and $-\frac{1}{4}\qd$.
Finally, $\what{\SW}^{-1}(h,h^\dual,K)=(c,|K|^{-1}\qd,K)$.

On the other hand, $\what{\SW}^{-1}(\lambda,\mu,-\infty)=(c,\qd,-\infty)$, where $\qd$ is the Hubbard-Masur $c$-holomorphic quadratic differential with $\cF_-(\qd)=\lambda$ and $\cF_+(\qd)=\mu$.

In order to show that $\what{\SW}$ is closed, we consider a sequence $\{(c_n,\qd_n,K_n)\}$ in $\cQ\times[-\infty,0)$
such that $(h_n,h_n^\dual,K_n)=\what{\SW}(c_n,\qd_n,K_n)$ 
converges and we want to show that $\{(c_n,\qd_n,K_n)\}$
has an accumulation point.
%% As the case $k_n\rar 0$ is obvious, we can distinguish two cases: either $k\in(0,+\infty)$ or $k=+\infty$.
Let $K=\lim_{n\rar \infty} K_n\in [-\infty,0)$. %% replacing previous line, I hope it's clearer.

Suppose that $K\in(-\infty,0)$ and $(h_n,h_n^\dual)\rar(h,h^\dual)\in\cT\times\cT$.
Because harmonic maps and minimal Lagrangian maps depend regularly on the metrics,
$c_n=[h_n+h_n^\dual]\rar c=[h+h^\dual]$ and $\qd_n\rar \qd$, where $\frac{\qd}{4|K|}$ is
the Hopf differential of the harmonic map $(S,c)\rar(S,h)$.

Suppose now that $K=-\infty$ and that
$|K_n|^{-1}(h_n,h^\dual_n)\rar(\lambda,\mu)\in\cFML$.
It follows from \cite{wolf:R-trees} that
the function $E(\bullet,|K_n|^{-1}h_n)$ converges
$C^\infty$ on the compact subsets of $\cT$ to
$\frac{1}{2}\Extr_\lambda(\bullet)$.
%For any $c'$ ranging in a compact subset of $\cT'$ we have
%$E(\bullet,\frac{1}{k_n}h_n)+E(\bullet,\frac{1}{k_n}h^\dual_n)\rar
%\frac{1}{2}(\mathrm{Ext}_\lambda(\bullet)+\mathrm{Ext}_\mu(\bullet))$
%
By Lemma~\ref{lemma:Ext}, the function
$\Extr_\lambda(\bullet)+\Extr_\mu(\bullet)$
achieves a unique minimum at the conformal structure $c$
underlying the quadratic differential $\qd$ with foliations $(\lambda,\mu)$.
By the above remark, the minima $c_n$ of $E(\bullet,|K_n|^{-1}h_n)+E(\bullet,|K_n|^{-1}h^\dual_n)$
converge to $c$ and
by \cite{wolf:R-trees} we conclude that $\qd_n\rar\qd$.
%
%As $c_n$ is the unique minimum for the function
%$E(\bullet,h_n)+E(\bullet,h_n^\dual)$, we have that
%$E(c_n,\frac{1}{k_n}h_n)+E(c_n,\frac{1}{k_n}h_n^\dual)
%\leq E(c',\frac{1}{k_n}h_n)+E(c',\frac{1}{k_n}h^\dual_n)\leq
%1+\frac{1}{2}(\mathrm{Ext}_\lambda(c')+\mathrm{Ext}_\mu(c'))$
%and so $E(c_n,\frac{1}{k_n}h_n)+E(c_n,\frac{1}{k_n}h_n^\dual)$
%is bounded. 
%
%This shows that $c_n$ is bounded too:
%by contradiction, if $(c_n)$ is divergent in $\cT$ (up to subsequences), then
%$E(c_n,\frac{1}{k_n}h_n)+E(c_n,\frac{1}{k_n}h^\dual_n)\rar+\infty$.
%% NB I'm also not sure why this is true, I think another ref would be needed.
%
%By \cite{wolf:teichmuller}, $q_n\rar q$, where $q$
%is the Hopf differential of the harmonic
%map $(S,c)\rar \lambda$ (and the opposite of the Hopf differential
%of the harmonic map $(S,c)\rar \mu$).
\end{proof}

\begin{proof}[Proof of Proposition~\ref{prop:action}]
The landslide flow on $\cT\times\cT$ can be extended to
$\cDY$ as $\Land_{e^{i\theta}}
(h,h^\dual,K)=(h_\theta,h^\dual_\theta,K)$. It is immediate to see that
$\what{\SW}$ conjugates this landslide flow on $\cDY$ with the flow
$e^{i\theta}\cdot(c_n,\qd_n,K)=(c_n,e^{i\theta}\qd_n,K)$
on $\cQ\times(-\infty,0)$ and
so it extends to $\cQ\times\{-\infty\}\cong\cFML\times\{-\infty\}$.
\end{proof}

\begin{proof}[Proof of Proposition~\ref{prop:ham-ext}]
The function $F$ on $\partial\cDY$ is given by
$F(\lambda,\mu)=2E(c,\lambda)=2E(c,\mu)=\|\qd\|$, where
$\frac{1}{4}\qd$ is the Hopf differential of the harmonic map from $c$ to the $\R$-tree dual to $\lambda$
and $\qd$ is the quadratic differential corresponding to $(\lambda,\mu)$
and $c$ is conformal structure underlying $\qd$, and so
$F(\lambda,\mu)=i(\lambda,\mu)$.

Using charts of $\cFML$ given by couples of maximal recurrent (and transversely recurrent) train tracks transverse to each other,
the symplectic form $\omega_{Th,1}+\omega_{Th,2}$ and the $1$-form $dF$ have constant coefficients and so define a local Hamiltonian flow (in charts). We want to show that this local flow
is exactly the limit of the landslide flow.

Notice that $F$ is real-analytic on $\cT\times\cT$ and so
it extends as a $C^1$ function to those points $\cFML_{max}$
of $\cFML$
that have a tangent space, namely to
couples $(\lambda,\mu)$ of maximal
measured laminations, which represent a dense
subset of full measure.

From \cite{papadopoulos-penner:TAMS} and \cite{bonahon-sozen:duke} it follows that %, up to a constant,
$K^{-2}\omega_{WP}$ on $\cY=\cT\times(-\infty,0)$
continuously extends
as Thurston's symplectic form $\omega_{Th}$ at those points
of $\cML\times\{-\infty\}$ that represent maximal measured laminations.

Thus, the vector field $\omega_{WP}^{-1}(\frac{1}{4}dF,-)$ that generates
the landslide flow converges almost everywhere
to $\omega_{Th}^{-1}(\frac{1}{4}dF,-)$. This implies that the landslide flow
converges locally uniformly to the flow locally
determined by $(\omega_{Th,1}+\omega_{Th,2})^{-1}(\frac{1}{4}dF,-)$
on $\cFML$.
%
%As the flow $(\theta,q)\mapsto e^{i\theta}q$ 
%is Hamiltonian for $\omega_{\cQ}$ with respect to the function
%$F=\|q\|$, the result follows.
\end{proof} %extension to the boundary (section 8)

\section{AdS geometry and composition of earthquakes}

\subsection{Dual constant curvature surfaces in AdS manifolds}

\begin{defi}
For any $K<-1$, set $K^\dual=-K/(K+1)$.  
\end{defi}

Our first goal is to prove the special case of Theorem \ref{tm:adsI} when
the curvatures of the future and past surfaces satisfy the relation
$K_+=K_-^\dual$. 

\begin{lemma}
Let $\mgh$ be a MGH AdS manifold, and let $S_+$ and $S_-$ be the surfaces in $\mgh$ with
curvature $K_+$ and $K_-$, respectively, in the future and in the past of the convex
core of $\mgh$. If $K_+=-K_-^\dual$ then $S_-$ is dual to $S_+$ (and conversely). 
If we identify $S_+$ to $S_-$ by the natural duality map, then the third 
fundamental form of $S_+$ is equal to the induced metric on $S_-$, and 
conversely.
\end{lemma}

\begin{proof}
It follows from Lemma \ref{lm:GHdual} that the surface dual to $S_+$ is a future
convex surface $S_+^\dual$ in the past of the convex core of $\mgh$. Point (3) of Lemma \ref{lm:dual} then
shows that $S_+^\dual$ has constant curvature $K_-=K_+^\dual$. But according to the main 
result of \cite{BBZ2}, there is a unique such space-like surface of constant curvature 
$K_-$ in the past of the convex core of $\mgh$, so $S_-=S_+^\dual$. Lemma \ref{lm:dual} then
shows that, under the identification of $S_+$ with $S_-$ by the duality map, the
induced metric on $S_+$ corresponds to the third fundamental form of $S_-$, and
conversely.
\end{proof}

The special case of Theorem \ref{tm:adsI} directly follows. 

\begin{lemma}  \label{lm:KK*}
Let $h_+, h_-\in \cT$, and let $K_+, K_-<-1$ with $K_+=K_-^\dual$. There exists 
a unique MGH AdS manifold $\mgh$ such that the past-convex surface of constant 
curvature $K_+$ in $\mgh$ is homothetic to $h_+$ while the future-convex surface of
constant curvature $K_-$ in $\mgh$ is homothetic to $h_-$.
\end{lemma}

\begin{proof}
Given two hyperbolic metrics $h_+, h_-\in \cT$ and two constants $K_-, K_+<-1$ such
that $K_-=K_+^\dual$, let $I=(-1/K_+)h_+, \III=(-1/K_-)h_-$. Consider the 
identification between $(S,h_-)$ and $(S, h_+)$ by the unique minimal Lagrangian
map isotopic to the identity, and let $b$ be the Labourie operator such that 
$h_-=h_+(b\bullet,b\bullet)$.

%% I've added what follows of the proof, it was incomplete
Let $k=\sqrt{-1-K_+}$, and set $B=kb$. Then $B$ is self-adjoint for $I$, solution
of the Codazzi equation for $I$, and of the AdS Gauss equation $\det(B)=-1-K_+$.
So there exists an equivariant embedding of the universal cover of $(S,I)$ as
a space-like, locally strictly convex surface in $\AdS^3$ with shape operator equal
to the lift of $B$ to $\St$. This implies that there is an isometric embedding of $(S,I)$ in a MGH 
AdS manifold $\mgh$, with shape operator equal also to $B$.

The properties of the duality map in $\AdS^3$ then imply that the surface $S^\dual$ dual to $S$
in $\mgh$ has induced metric equal to $\III$, in particular it has constant curvature $K_-$ and
is homothetic to $h_-$. 
This already shows the existence of $\mgh$ containing the required surfaces. 

The uniqueness of $\mgh$ follows from the same arguments, and from the fact that any MGH 
AdS manifolds contains a unique past-convex and a unique future-convex surface of any
given curvature in $(-\infty,-1)$, see \cite{BBZ2}, so that for any $K\in (-\infty,-1)$,
the past-convex surface of constant curvature $K^\dual$ is always dual to the future-convex
surface of constant curvature $K$.
\end{proof}

\subsection{AdS manifolds with constant curvature boundary}

The more general part of Theorem \ref{tm:adsI} will follow from a compactness
argument. We will need the following elementary statement on the Teichm\"uller distance.
Given a hyperbolic metric $h$ and a closed curve $\gamma$ on $S$, we denote by 
$\ell_\gamma(h)$ the length of the geodesic for $h$ homotopic to $\gamma$.
% We also denote by $d_T$ the Teichm\"uller distance on $\cT$.

\begin{lemma} \label{lm:compact}
Let $\const>1$ and $h\in \cT$. 
\begin{enumerate}
\item The set of hyperbolic metrics $h'$ on $S$ such that, for all closed curve
$\gamma$ on $S$, $\ell_{\gamma}(h') \leq \const\,\ell_{\gamma}(h)$, is compact. 
\item Similarly, the set of metrics $h'$ on $S$ such that, for all closed curves $\gamma$, 
$\ell_{\gamma}(h)\leq \const\,\ell_{\gamma}(h')$, is compact.
\end{enumerate}
\end{lemma}

\begin{proof}
Recall that Thurston's asymetric distance $d_{Th}(h,h')$ between $h$ and $h'$ is defined as the log of the
infimum of the Lipschitz constants over all smooth maps from $(S,h)$ to $(S,h')$ isotopic
to the identity (see \cite{thurston:minimal}). It can also be defined as the supremum of the
ratio of length for $h$ and for $h'$ of closed curves on $S$, see \cite{thurston:minimal}.
It is known (see \cite{papadopoulos-theret}) that, if $h$ is fixed,
then $d_{Th}(h,h'_n)\rightarrow \infty$
as $h'_n\rightarrow \infty$. This proves the first point. Similarly, if $h'$ is fixed
and $h_n\rightarrow \infty$, then $d_{Th}(h_n,h')\rightarrow \infty$, and this proves the second
point.
\end{proof}

\begin{cor} \label{cr:compact}
Let $\const>1$ and $\cC\subset \cT$ be compact.
Let $\cC'$ be the set of all metrics $h'\in \cT$ such
that $\ell_{\gamma}(h')
\leq \const\,\ell_{\gamma}(h)$ (resp. $\ell_{\gamma}(h) \leq \const \,\ell_{\gamma}(h')$) for some $h\in \cC$ and for all closed curve $\gamma$ on $S$. Then $\cC'$ is compact.
\end{cor}

This corollary will be useful in conjunction which the following basic estimate
from AdS geometry.

\begin{lemma} 
Let $\mgh$ be a MGH AdS manifold, let $K<K'<-1$, and let $S,S'$ be the future-convex
surfaces of constant curvature $K$ and $K'$, respectively, in $\mgh$. Let $\gamma$
be a closed geodesic in $S$. Then the length of $\gamma$ is smaller than the length
of the closed geodesic $\gamma'$ in $S'$ homotopic to $\gamma$.
\end{lemma}

\begin{proof}
This follows from the elementary fact that, in a foliation of an AdS manifold
by future-convex surfaces (identified by the normal flow), the metric is
decreasing when moving towards the past, see e.g. \cite{colI}.
\end{proof}

\begin{cor} \label{cr:comparison}
Let $K_-, K_+<-1$ with $K_-<K_+^\dual$. Let $\mgh$ be a MGH AdS manifold containing a past-convex 
surface $S_+$ with induced metric $(-1/K_+)h_+$ and third fundamental form $(-1/K_+^\dual)h_+^\dual$,
so that $h_+$ and $h_+^\dual$ are hyperbolic metrics. Let $h_-$ (resp. $h_-^\dual$) be the hyperbolic metric 
homothetic to the induced metric (resp. third fundamental form) of the future-convex surface $S_-$ 
of constant curvature $K_-$. Then 
$$ h_- \leq \left( \frac {K_-}{K_+^\dual}\right) h_+^\dual~, 
\qquad
h_+ \leq \left( \frac {K_+}{K_-^\dual}\right) h_-^\dual~, $$
where the inequalities are understood in the sense of the length spectrum. %% added this sentence
Similarly if $K_+<K_-^\dual$ then 
$$ h_+ \leq \left( \frac {K_+}{K_-^\dual}\right) h_-^\dual~,
\qquad
h_- \leq \left( \frac {K_-}{K_+^\dual}\right) h_+^\dual~. $$
\end{cor}

\begin{defi}
Let $K_-, K_+<-1$. We denote by $\Phi_{K_-,K_+}:\cT\times \cT\rightarrow \cT\times \cT$
the map sending $(h_l, h_r)$ to the hyperbolic metrics $h_-, h_+$ such that the 
MGH AdS manifold $\mgh$ with left and right metrics $h_l$ and $h_r$ contains a 
past-convex surface of constant curvature $K_+$ with induced metric $(-1/K_+)h_+$, 
and a future-convex surface of constant curvature $K_-$ with induced metric
$(-1/K_-)h_-$. 
\end{defi}

It follows from Lemma \ref{lm:KK*} that $\Phi_{K_-,K_-^\dual}$ is a homeomorphism
for all $K_-<-1$. To prove Theorem \ref{tm:adsI}, we will show that $\Phi_{K_-,K_+}$
is ``bounded'' in a suitable sense by $\Phi_{K_-,K_-^\dual}$ or $\Phi_{K_+^\dual,K_+}$.

%% The following statement follows directly from Corollary \ref{cr:comparison} and
%% Corollary \ref{cr:compact}.
%% I've suppressed the following sentence since I think it was not helpful.

\begin{cor}
For all $K_-, K_+<-1$, $\Phi_{K_-, K_+}$ is proper.
\end{cor}

\begin{proof}
We consider two cases, depending on whether $K_-$ is smaller or larger than $K_+^\dual$.
Assume first that $K_-<K_+^\dual$. Let $\cDC\subset \cT\times \cT$ be compact, and 
let $\cC_-, \cC_+$ be two compact subsets of $\cT$ such that $\cDC\subset \cC_-\times \cC_+$.
Suppose that $\Phi_{K_-,K_+}(h_l,h_r)=(h_-,h_+)\in\cDC$.
Then $\Phi_{K_-,K_-^\dual}(h_l,h_r)=(h_-, h_-^\dual)$ with
$h_+ \leq \left( K_+/K_-^\dual\right) h_-^\dual$ by Corollary \ref{cr:comparison}. It follows that $h_-^\dual$ is in a compact set $\cC^\dual_-$ which depends
only on $\cC_+$ and on $K_+/K_-^\dual$ by Corollary \ref{cr:compact}.

Since $\Phi_{K_-,K_-^\dual}$ is a homeomorphism, $\Phi_{K_-,K_-^\dual}^{-1}(\cC_-\times \cC_-^\dual)$ is 
a compact subset $\cDC'$ of $\cT\times \cT$. By construction, $(h_l, h_r)\in \cDC'$
whenever $(h_-,h_+)\in \cDC$. This shows that $\Phi_{K_-,K_+}$ is proper. %% I've cut the sentence in two.

The same argument proves the same result when $K_->K_+^\dual$, except that now
$h_+ \leq \left( {K_+}/{K_-^\dual}\right) h_-^\dual$ and the other inequality has to be
used in Corollary \ref{cr:compact}.
\end{proof}

\begin{proof}[Proof of Theorem \ref{tm:adsI}]
As $\Phi_{K_-,K_+}$ is proper, its degree is well-defined for all $K_-,K_+<-1$. Moreover, it easily follows from the above corollary that, for every $K_{min}<K_{max}<-1$, the map
$\Phi:[K_{min},K_{max}]^2\times\cT\times\cT\rar\cT\times\cT$
defined as $\Phi(K_-,K_+,h_l,h_r):=\Phi_{K_-,K_+}(h_l,h_r)$
is proper. Hence, the degree of $\Phi_{K_-,K_+}$ does not depend on the chosen $(K_-,K_+)$, and in particular it coincides with
the degree of $\Phi_{K_-,K_-^\dual}$.
But we already know that $\Phi_{K_-,K_-^\dual}$ is a homeomorphism. Hence, for all $K_-,K_+<-1$ the map $\Phi_{K_-,K_+}$ has degree $1$ and
so it is onto. This proves the theorem.
\end{proof}

\subsection{Prescribed third fundamental forms}

Theorem \ref{tm:adsIII} follows from Theorem \ref{tm:adsI} through the duality
between constant curvature surfaces in MGH AdS manifolds (see Lemma \ref{lm:GHdual}).
In particular,
if $\mgh$ is a MGH AdS manifold containing a past-convex surface $S_+$ of constant
curvature $K_+$ with induced metric homothetic to $h_+$ and a future-convex
surface $S_-$ of constant curvature $K_-$ and induced metric homothetic to $h_-$,
then the surface $S_+^\dual$ dual to $S_+$ is future-convex, has constant curvature
$K_+^\dual$ and third fundamental form homothetic to $h_+$, while the surface
$S_-^\dual$ dual to $S_-$ has constant curvature $K_-^\dual$ and third fundamental form
homothetic to $h_-$.

\subsection{Fixed points of compositions of landslides}

We now turn to the proof of Theorem \ref{tm:fixed}. 

The relationship between constant curvature surfaces in MGH AdS manifolds and
landslides is captured in the following statement, strongly analoguous to 
a well-known statement for earthquakes, see \cite{mess,mess-notes,earthquakes}.

\begin{lemma} \label{lm:diagram}
Let $\mgh$ be a MGH AdS manifold, with left and right hyperbolic metrics $h_l, h_r$. 
Let $K_-, K_+<-1$, and let $S_-, S_+$ be the future-convex and past-convex 
surfaces with constant curvature $K_-$ and $K_+$, respectively. 
Let $h_-, h_+$ (resp. $h_-^\dual, h_+^\dual$) be the hyperbolic metrics homothetic to 
the induced metrics (resp. third fundamental forms) on $S_-$ and $S_+$, respectively.
Then
\begin{alignat}{2}
  \label{eq:diagram+}
   h_l  & = \Land^1_{e^{it_+}}(h_+, h_+^\dual)~, &
\qquad h_r  & = \Land^1_{e^{-it_+}}(h_+, h_+^\dual)~, \\
  \label{eq:diagram-}
h_l & = \Land^1_{e^{-it_-}}(h_-, h_-^\dual)~, &
\qquad h_r & = \Land^1_{e^{it_-}}(h_-, h_-^\dual)~, 
\end{alignat}
%% corrected the t_+ in e^{it_+} in the four equations to be coherent with notations.
where $K_+=-1/\cos^2(t_+/2), K_-=-1/\cos^2(t_-/2)$. 

Conversely, if (\ref{eq:diagram+}) and (\ref{eq:diagram-}) are satisfied then
there exists a MGH AdS manifold $\mgh$ with left and right hyperbolic metrics 
$h_l, h_r$, containing a past-convex surface $S_+$ with constant curvature $K_+$
and induced metric and third fundamental form homothetic to $h_+$ and $h_+^\dual$,
and a future convex surface $S_-$ with constant curvature $K_-$
and induced metric and third fundamental form homothetic to $h_-$ and $h_-^\dual$.
\end{lemma}

\begin{proof} %% I rewrote this proof with slightly more details.
The first point follows directly from \cite[Lemma 1.9]{cyclic}. 
The converse also follows from the same lemma, because a MGH AdS manifold is uniquely
determined by its left and right hyperbolic metrics (see \cite{mess}) so that the 
MGH AdS manifold containing a past-convex space-like surface of curvature $K_+$ with
$I$ and $\III$ respectively homothetic to $h_+$ and $h_+^\dual$ is the same as the 
MGH AdS manifold containing a future-convex space-like surface of curvature $K_-$
with $I$ and $\III$ respectively homothetic to $h_-$ and $h_-^\dual$.
\end{proof}

\begin{proof}[Proof of Theorem \ref{tm:fixed}] %% I've added this proof
Let $\theta_-,\theta_+\in (0,\pi)$, and let $h_-,h_+\in \cT$. 
Set $K_+=-1/\cos^2(\theta_+/2)$ and $K_-=-1/\cos^2(\theta_-/2)$. 
Theorem \ref{tm:adsI} indicates that there exists a MGH AdS manifold $\mgh$ containing a 
past-convex space-like surface of constant curvature $K_+$ proportional to $h_+$, 
and a future-convex space-like surface of constant curvature $K_-$ proportional to 
$h_-$. Moreover, if $\theta_-+\theta_+=\pi$, then $\mgh$ is unique.

Let $h_l, h_r$ be the left and right hyperbolic metrics of $\mgh$. 
Lemma \ref{lm:diagram} then shows that $h_r=\Sy_{e^{i\theta_+},h_+}(h_l)$, while 
$h_l=\Sy_{e^{i\theta_-},h_-}(h_r)$. Thus $h_r$ is a fixed point of 
$\Sy_{e^{i\theta_+},h_+}\circ \Sy_{e^{i\theta_-},h_-}$. This proves the existence part of the statement. 

The uniqueness part when $\theta_-+\theta_+=\pi$ follows from the uniqueness of $\mgh$ in 
this case, together with the converse part of Lemma \ref{lm:diagram}.
\end{proof} %AdS K-surfaces and fixed points (section 9)

\bibliographystyle{amsplain}
%\bibliography{biblio,bibcyclic}

\begin{thebibliography}{10}

\bibitem{mess-notes}
Lars Andersson, Thierry Barbot, Riccardo Benedetti, Francesco Bonsante,
  William~M. Goldman, Fran{\c{c}}ois Labourie, Kevin~P. Scannell, and Jean-Marc
  Schlenker, \emph{Notes on: ``{L}orentz spacetimes of constant curvature''
  [{G}eom. {D}edicata {\bf 126} (2007), 3--45; mr2328921] by {G}. {M}ess},
  Geom. Dedicata \textbf{126} (2007), 47--70. \MR{MR2328922}

\bibitem{adsquestions}
T.~Barbot, F.~Bonsante, J.~Danciger, W.M. Goldman, F.~Gu{\'e}ritaud, F.~Kassel,
  K.~Krasnov, J.M. Schlenker, and A.~Zeghib, \emph{Some open questions on
  anti-de sitter geometry}, Arxiv preprint arXiv:1205.6103 (2012).

\bibitem{BBZ2}
Thierry Barbot, Fran{\c{c}}ois B{\'e}guin, and Abdelghani Zeghib,
  \emph{Prescribing {Gauss} curvature of surfaces in 3-dimensional spacetimes,
  application to the {Minkowski} problem in {Minkowski} space}, Ann. Inst.
  Fourier (Grenoble) \textbf{61} (2011), no.~1, 511Ð591.

\bibitem{colI}
Thierry Barbot, Francesco Bonsante, and Jean-Marc Schlenker, \emph{Collisions
  of particles in locally {AdS} spacetimes {I}. {Local} description and global
  examples}, arXiv:1010.3602. {\it Comm. Math. Phys.} 308(2011):1, 147-200.,
  2011.

\bibitem{cyclic}
F.~{Bonsante}, G.~{Mondello}, and J.-M. {Schlenker}, \emph{{A cyclic extension
  of the earthquake flow}}, ArXiv e-prints (2011).

\bibitem{earthquakes}
Francesco Bonsante and Jean-Marc Schlenker, \emph{Fixed points of compositions
  of earthquakes}, arXiv:0812.3471. To appear, {\it Duke Math. J.}, 2009.

\bibitem{dumas-wolf}
David Dumas and Michael Wolf, \emph{Projective structures, grafting and
  measured laminations}, Geom. Topol. \textbf{12} (2008), no.~1, 351--386.
  \MR{MR2390348 (2009c:30114)}

\bibitem{eells-lemaire:deformation}
J.~Eells and L.~Lemaire, \emph{Deformations of metrics and associated harmonic
  maps}, Proc. Indian Acad. Sci. Math. Sci. \textbf{90} (1981), no.~1, 33--45.
  \MR{653945 (83g:58013)}

\bibitem{FLP}
A.~Fathi, F.~Laudenbach, and V.~Poenaru, \emph{Travaux de {T}hurston sur les
  surfaces}, Soci\'et\'e Math\'ematique de France, Paris, 1991, S\'eminaire
  Orsay, Reprint of {\it Travaux de Thurston sur les surfaces}, Soc.\ Math.\
  France, Paris, 1979 [MR 82m:57003], Ast\'erisque No. 66-67 (1991).

\bibitem{fischer-tromba:cell}
A.~E. Fischer and A.~J. Tromba, \emph{A new proof that {T}eichm\"uller space is
  a cell}, Trans. Amer. Math. Soc. \textbf{303} (1987), no.~1, 257--262.
  \MR{896021 (89b:32030)}

\bibitem{gardiner:acta84}
Frederick~P. Gardiner, \emph{Measured foliations and the minimal norm property
  for quadratic differentials}, Acta Math. \textbf{152} (1984), no.~1-2,
  57--76. \MR{736212 (85i:30085)}

\bibitem{graham-witten}
C.~R. Graham and E.~Witten, \emph{Conformal anomaly of submanifold observables
  in {AdS/CFT} correspondence}, hep-th/9901021.

\bibitem{HR}
Craig~D. Hodgson and Igor Rivin, \emph{A characterization of compact convex
  polyhedra in hyperbolic 3-space}, Invent. Math. \textbf{111} (1993), 77--111.

\bibitem{hubbard-masur:acta79}
John Hubbard and Howard Masur, \emph{Quadratic differentials and foliations},
  Acta Math. \textbf{142} (1979), no.~3-4, 221--274. \MR{523212 (80h:30047)}

\bibitem{kamishima-tan}
Yoshinobu Kamishima and Ser~P. Tan, \emph{Deformation spaces on geometric
  structures}, Aspects of low-dimensional manifolds, Adv. Stud. Pure Math.,
  vol.~20, Kinokuniya, Tokyo, 1992, pp.~263--299. \MR{1208313 (94k:57023)}

\bibitem{kawai}
Shingo Kawai, \emph{The symplectic nature of the space of projective
  connections on {R}iemann surfaces}, Math. Ann. \textbf{305} (1996), no.~1,
  161--182. \MR{MR1386110 (97a:32015)}

\bibitem{kerckhoff}
Steven~P. Kerckhoff, \emph{The {N}ielsen realization problem}, Ann. of Math.
  (2) \textbf{117} (1983), no.~2, 235--265. \MR{MR690845 (85e:32029)}

\bibitem{kerckhoff:minima}
\bysame, \emph{Lines of minima in {T}eichm\"uller space}, Duke Math. J.
  \textbf{65} (1992), no.~2, 187--213. \MR{1150583 (93b:32027)}

\bibitem{minsurf}
Kirill Krasnov and Jean-Marc Schlenker, \emph{Minimal surfaces and particles in
  3-manifolds}, Geom. Dedicata \textbf{126} (2007), 187--254. \MR{MR2328927}

\bibitem{volume}
\bysame, \emph{On the renormalized volume of hyperbolic 3-manifolds}, Comm.
  Math. Phys. \textbf{279} (2008), no.~3, 637--668. \MR{MR2386723}

\bibitem{cp}
Kirill Krasnov and Jean-Marc Schlenker, \emph{A symplectic map between
  hyperbolic and complex {Teichm\"uller} theory}, arXiv:0806.0010. {\it Duke
  Math. J.} 150(2009):2, 331-356, 2008.

\bibitem{review}
\bysame, \emph{The {Weil-Petersson} metric and the renormalized volume of
  hyperbolic 3-manifolds}, arXiv:0907.2590. To appear, \emph{Handbook of
  Teichm\"uller theory}, vol. III., 2009.

\bibitem{L5}
Fran\c{c}ois Labourie, \emph{Surfaces convexes dans l'espace hyperbolique et
  {CP1}-structures}, J. London Math. Soc., II. Ser. \textbf{45} (1992),
  549--565.

\bibitem{mcmullen:complex}
Curtis~T. McMullen, \emph{Complex earthquakes and {T}eichm\"uller theory}, J.
  Amer. Math. Soc. \textbf{11} (1998), no.~2, 283--320. \MR{1478844
  (98i:32030)}

\bibitem{mess}
Geoffrey Mess, \emph{Lorentz spacetimes of constant curvature}, Geom. Dedicata
  \textbf{126} (2007), 3--45. \MR{MR2328921}

\bibitem{papadopoulos-penner:TAMS}
A.~Papadopoulos and R.~C. Penner, \emph{The {W}eil-{P}etersson symplectic
  structure at {T}hurston's boundary}, Trans. Amer. Math. Soc. \textbf{335}
  (1993), no.~2, 891--904. \MR{1089420 (93d:57022)}

\bibitem{papadopoulos-theret}
Athanase Papadopoulos and Guillaume Th{\'e}ret, \emph{On the topology defined
  by {T}hurston's asymmetric metric}, Math. Proc. Cambridge Philos. Soc.
  \textbf{142} (2007), no.~3, 487--496. \MR{2329697 (2008m:57046)}

\bibitem{sem-era}
Igor Rivin and Jean-Marc Schlenker, \emph{The {Schl\"afli} formula in
  {Einstein} manifolds with boundary}, Electronic Research Announcements of the
  A.M.S. \textbf{5} (1999), 18--23.

\bibitem{sem}
\bysame, \emph{The {Schl\"afli} formula and {Einstein} manifolds}, Preprint
  math.DG/0001176, 2000.

\bibitem{sampson:78}
J.~H. Sampson, \emph{Some properties and applications of harmonic mappings},
  Ann. Sci. \'Ecole Norm. Sup. (4) \textbf{11} (1978), no.~2, 211--228.
  \MR{510549 (80b:58031)}

\bibitem{scannell-wolf}
Kevin~P. Scannell and Michael Wolf, \emph{The grafting map of {T}eichm\"uller
  space}, J. Amer. Math. Soc. \textbf{15} (2002), no.~4, 893--927 (electronic).
  \MR{MR1915822 (2003d:32011)}

\bibitem{these}
Jean-Marc Schlenker, \emph{Surfaces convexes dans des espaces lorentziens \`a
  courbure constante}, Comm. Anal. Geom. \textbf{4} (1996), no.~1-2, 285--331.
  \MR{MR1393565 (98c:53076)}

\bibitem{hmcb}
\bysame, \emph{Hyperbolic manifolds with convex boundary}, Invent. Math.
  \textbf{163} (2006), no.~1, 109--169. \MR{MR2208419 (2006m:57023)}

\bibitem{schoen:role}
Richard~M. Schoen, \emph{The role of harmonic mappings in rigidity and
  deformation problems}, Complex geometry ({O}saka, 1990), Lecture Notes in
  Pure and Appl. Math., vol. 143, Dekker, New York, 1993, pp.~179--200.
  \MR{MR1201611 (94g:58055)}

\bibitem{bonahon-sozen:duke}
Ya{\c{s}}ar S{\"o}zen and Francis Bonahon, \emph{The {W}eil-{P}etersson and
  {T}hurston symplectic forms}, Duke Math. J. \textbf{108} (2001), no.~3,
  581--597. \MR{1838662 (2002c:32023)}

\bibitem{strebel:quadratic}
Kurt Strebel, \emph{Quadratic differentials}, Ergebnisse der Mathematik und
  ihrer Grenzgebiete (3) [Results in Mathematics and Related Areas (3)],
  vol.~5, Springer-Verlag, Berlin, 1984. \MR{743423 (86a:30072)}

\bibitem{TZ-schottky}
L.~Takhtajan and P.~Zograf, \emph{On uniformization of {Riemann} surfaces and
  the {Weil-Petersson} metric on the {Teichm\"uller} and {Schottky} spaces},
  Mat. Sb. \textbf{132} (1987), 303--320, English translation in {\it Math.
  USSR Sb. 60:297-313, 1988}.

\bibitem{takhtajan-zograf:spheres}
Leon Takhtajan and Peter Zograf, \emph{Hyperbolic 2-spheres with conical
  singularities, accessory parameters and {K}\"ahler metrics on {$M\sb
  {0,n}$}}, Trans. Amer. Math. Soc. \textbf{355} (2003), no.~5, 1857--1867
  (electronic). \MR{MR1953529 (2003j:32031)}

\bibitem{takhtajan-teo}
Leon~A. Takhtajan and Lee-Peng Teo, \emph{Liouville action and
  {W}eil-{P}etersson metric on deformation spaces, global {K}leinian
  reciprocity and holography}, Comm. Math. Phys. \textbf{239} (2003), no.~1-2,
  183--240. \MR{MR1997440 (2005c:32021)}

\bibitem{thurston-earthquakes}
William~P. Thurston, \emph{Earthquakes in two-dimensional hyperbolic geometry},
  Low-dimensional topology and Kleinian groups (Coventry/Durham, 1984), London
  Math. Soc. Lecture Note Ser., vol. 112, Cambridge Univ. Press, Cambridge,
  1986, pp.~91--112. \MR{MR903860 (88m:57015)}

\bibitem{thurston:minimal}
W.P. Thurston, \emph{Minimal stretch maps between hyperbolic surfaces}, Arxiv
  preprint math/9801039 (1998).

\bibitem{wolf:teichmuller}
Michael Wolf, \emph{The {T}eichm\"uller theory of harmonic maps}, J.
  Differential Geom. \textbf{29} (1989), no.~2, 449--479. \MR{982185
  (90h:58023)}

\bibitem{wolf:R-trees}
\bysame, \emph{Harmonic maps from surfaces to {$\bold R$}-trees}, Math. Z.
  \textbf{218} (1995), no.~4, 577--593. \MR{1326987 (97b:58042)}

\end{thebibliography}

\def\cprime{$'$} \def\cprime{$'$}
\providecommand{\bysame}{\leavevmode\hbox to3em{\hrulefill}\thinspace}
\providecommand{\MR}{\relax\ifhmode\unskip\space\fi MR }
\providecommand{\MRhref}[2]{%
  \href{http://www.ams.org/mathscinet-getitem?mr=#1}{#2}
}
\providecommand{\href}[2]{#2}

\end{document}